\documentclass[10pt]{article}

\usepackage{amscd,amsmath, amssymb, fancyhdr, mathbbol, dutchcal, url}
\usepackage[matrix,arrow,curve]{xy}
\usepackage[backref=page]{hyperref}
\renewcommand*{\backrefalt}[4]{%
	\ifcase #1 (Not cited.)%
	\or        (Cited on page~#2.)%
	\else      (Cited on pages~#2.)%
	\fi}

\hypersetup{
	colorlinks   = true,
	citecolor    = magenta,
	linkcolor    =blue,
	urlcolor     =magenta	
}

\numberwithin{equation}{section}

\newcommand{\version}{version 1.1,\ \ June 29, 2020}

\def\eqref#1{(\ref{#1})}

\newcommand{\g}{{\mathfrak g}}

\def\C{{\Bbb C}}
\newcommand{\R}{{\Bbb R}}

\renewcommand{\H}{{\Bbb H}}

\renewcommand{\c}{\mathfrak c}
\def\1{\sqrt{-1}\:}
\newcommand{\restrict}[1]{{\left|_{{\phantom{|}\!\!}_{#1}}\right.}}
\newcommand{\cntrct}                % contraction with a vector field
{\hspace{2pt}\raisebox{1pt}{\text{$\lrcorner$}}\hspace{2pt}}
\newcommand{\twice}{/\!\!/}
\newcommand{\trice}{/\!\!/\!\!/}

\newcommand{\calM}{\mathcal M}
\newcommand{\calO}{\mathcal O}

\renewcommand{\P}{\mathbb P}

\renewcommand{\tilde}{\widetilde}
\renewcommand{\bar}{\overline}
\renewcommand{\phi}{\varphi}
\renewcommand{\epsilon}{\varepsilon}

\newcommand{\<}{\langle}
\renewcommand{\>}{\rangle}

\newcommand{\im}{\operatorname{im}}
\newcommand{\End}{\operatorname{End}}
\newcommand{\Ext}{\operatorname{Ext}}
\newcommand{\Tot}{\operatorname{Tot}}

\newcommand{\id}{\operatorname{\text{\sf id}}}

\newcommand{\diag}{\operatorname{\sf diag}}

\newcommand{\Tw}{\operatorname{Tw}}

\renewcommand{\Re}{\operatorname{Re}}
\renewcommand{\Im}{\operatorname{Im}}

\newcounter{Mycounter}[section]
\newcounter{lemma}[section]
\newcounter{claim}[section]
\newcounter{sublemma}[section]
\newcounter{corollary}[section]
\newcounter{theorem}[section]
\newcounter{conjecture}[section]
\newcounter{proposition}[section]
\newcounter{definition}[section]
\renewcommand{\thedefinition}
      {{Definition~\thesection.\arabic{definition}}}
\newcommand{\definition}{%
    \setcounter{definition}{\value{Mycounter}}
    \refstepcounter{definition}
    \stepcounter{Mycounter}
    {\noindent \bf \thedefinition:\ }}

\newcounter{example}[section]
\newcounter{remark}[section]
\renewcommand{\theremark}{{Remark \thesection.\arabic{remark}}}
\newcommand{\remark}{%
    \setcounter{remark}{\value{Mycounter}}
    \refstepcounter{remark}
    \stepcounter{Mycounter}
    {\noindent \bf \theremark:\ }}

\newcounter{problem}[section]
\newcounter{question}[section]
\newcommand{\pstep}{\noindent{\bf Proof. Step 1:\ }}

\makeatletter

\renewcommand{\leftmark}{{\scriptsize Feix--Kaledin metric on cotangent bundles to K\"ahler quotients}}
\renewcommand{\rightmark}{{\small 
A. Abasheva}}

\@addtoreset{equation}{section} \@addtoreset{footnote}{section}
\makeatother

\def\blacksquare{\hbox{\vrule width 5pt height 5pt depth 0pt}}
\def\endproof{\blacksquare}

\begin{document}
\begin{center}
{\LARGE\bf
Feix--Kaledin metric on the total spaces of cotangent bundles to K\"ahler quotients\\[4mm]
}
%%%%%%%%%%%%%%%%%%%%%%%%%%%%%%%%%%%%%%%%%%%%%%%%%%%%%%%%%%%%%

Anna Abasheva%\footnote{nothing}

\end{center}
{\small \hspace{0.10\linewidth}
\begin{minipage}[t]{0.85\linewidth}
{\bf Abstract.}
In this paper we study the geometry of the total space $Y$ of a cotangent bundle to a K\"ahler manifold $N$ where $N$ is obtained as a K\"ahler reduction from $\C^n$. Using the hyperk\"ahler reduction we construct a hyperk\"ahler metric on $Y$ and prove that it coincides with the canonical Feix--Kaledin metric. This metric is in general non-complete. We show that the metric completion $\tilde Y$ of the space $Y$ is equipped with a structure of a stratified hyperk\"ahler space (in the sense of \cite{Mayrand_stratification}). We give a necessary condition for the Feix--Kaledin metric to be complete using an observation of R. Bielawski. Pick a complex structure $J$ on $\tilde Y$ induced from the quaternions. Suppose that $J\ne\pm I$ where $I$ is the complex structure whose restriction to $Y = T^*N$ is induced by the complex structure on $N$. We prove that the space $\tilde{Y}_J$ admits an algebraic structure and is an affine variety.

\end{minipage}
}

\tableofcontents

\section{Introduction}
%%% list of changes in v.1.1 %%%
% 1. Added a remark on hK/qK-correspondence
%%%% list of changes on v.1.2 %%%%
% 1. Change a notation for a hyperk\"ahler metric from h to g in Definition 1.1. and Theorem 1.2.
% 2. Added to the statement of Theorem 1.5 the result about bigness.
% 3. Some grammar and style.
%%%%%%

Many known examples of non-compact hyperk\"ahler manifolds arise as total spaces of cotangent bundles to certain K\"ahler manifolds or contain them as an open dense subset (\cite{Kovalev}, \cite{_BiquardG_}, \cite{Bielawski_Dancer}, \cite{Nakajima_Takayama} ). Moreover, all these examples are equipped with a natural $U(1)$-action compatible with the hyperk\"ahler structure in the sense of the following definition.

\hfill

\definition
\label{compatible_action}
Let $X$ be a hyperk\"ahler manifold with the hyperk\"ahler metric $g$ and the holomorphic symplectic form $\Omega\in \Lambda^{2,0}M$. A holomorphic $U(1)$-action on $X$ is said to be {\em HKLR-compatible} with the hyperk\"ahler structure if:
\begin{itemize}
\item[(i)] The metric $g$ is $U(1)$-invariant.
\item[(ii)] $\mathsf L_{\phi}\Omega = \sqrt{-1}\Omega$ where $\phi$ is the vector field tangent to the action (equivalently, $\lambda^*\Omega = \lambda\Omega$ for every $\lambda\in U(1)$).
\item[(iii)] The fixed-point set of the $U(1)$-action is a complex Lagrangian submanifold.
\end{itemize}

\hfill

\remark The $U(1)$-action defined above does not preserve all the complex structures but rotates them. Circle actions preserving all the complex structures will not be considered in this paper. However, they arise in the context of hK/qK-correspondence (\cite{Haydys}) which links both types of circle actions with quaternionic K\"ahler geometry.

\hfill

A universal way to construct hyperk\"ahler manifolds with a HKLR-compatible $U(1)$-action is provided by the following theorem proved independently by D. Kaledin and B. Feix.

\hfill

\begin{theorem}
\label{Feix_Kaledin_preliminary}
(\cite{_Feix1_},\cite{_Feix2_},\cite{_Kaledin1_},\cite{_Kaledin2_}) Let $M$ be a K\"ahler manifold such that the corresponding K\"ahler metric is real analytic.  Then there exists a $U(1)$-invariant neighbourhood $X$ of the zero section of $T^*M$ which admits a hyperk\"ahler structure (\ref{hk}) $(g,I,J,K)$ satisfying the following properties:

\begin{itemize}
\item[(i)] Consider the natural complex structure $I$ on $T^*M$ induced by the complex structure on $M$. Then the holomorphic symplectic form $\Omega\in\Lambda^{2,0}_I X$ induced by the hyperk\"ahler structure is the restriction to $X$ of the canonical holomorphic symplectic form on $T^*M$ (formula~(\ref{standard_two_form})).
\item[(ii)] The natural $U(1)$-action on $X$ by fiberwise rotations is HKLR-compatible with the hyperk\"ahler structure.
\item[(iii)] The K\"ahler structure $(g,I)$ restricts to the given K\"ahler structure on $M$.
\end{itemize}

Moreover this hyperk\"ahler structure is unique up to a linear automorphism of $X$.
\end{theorem}

\hfill

\begin{definition} We shall call the hyperk\"ahler structure on $X$ uniquely defined by \ref{Feix_Kaledin_preliminary} the {\em Feix--Kaledin hyperk\"ahler structure}.
\end{definition}

\hfill

\remark A nice construction of the Feix--Kaledin hyperk\"ahler structure in a more general setting was given by A. Bor\'owka and D. Calderbank in \cite{Borowka_Calderbank}. Their method is close to the one of B. Feix. A simpler and more explicit geometric approach was discovered by R. Bielawski in \cite{Bielawski_complexification}. However, his methods are limited to a certain class of K\"ahler manifolds, namely complexifications of Riemannian manifolds.

\hfill

Our goal in this paper is to study a series of examples of K\"ahler manifolds $M$ such that $T^*M$ admits globally defined Feix--Kaledin metric. These examples are obtained via K\"ahler reduction from finite dimensional unitary representations. We shall postpone the description of their construction to the next section and now formulate the main theorem.

\hfill

\begin{theorem}
\label{mmm}
Let $V$ be a unitary representation of a compact group $C$ and $T^*V$ the total space of its cotangent bundle with the natural hyperk\"ahler metric. Denote by $N:=V\twice C$ the K\"ahler quotient of $V$ by the Hamiltonian action of $C$. The following assertions hold:
\begin{itemize}
\item[(i)] The Feix--Kaledin hyperk\"aler metric on $T^*N$ is defined globally.
\item[(ii)] There exists a natural open embedding of the hyperk\"ahler manifold $T^*N$ to the hyperk\"ahler quotient $(T^*V)\trice C$, moreover, the last variety is complete as a metric space.
\end{itemize}
\end{theorem}

\hfill

We are able to give a necessary condition for the Feix--Kaledin metric on $T^*N$ to be complete.

\hfill

\begin{theorem}
Let $N = V\twice C$ be a K\"ahler quotient of a vector space $V$. Suppose that $N$ is smooth and compact. If the Feix--Kaledin metric on $Y = T^*N$ is complete then the tangent bundle $TN$ to the complex manifold $N$ is big and nef.
\end{theorem}

\hfill

All the fibers of the twistor space of $(T^*V)\trice C$ except for those over $\pm I$ are isomorphic. This is true for every complete hyperk\"ahler manifold equipped with a HKLR-compatible $U(1)$-action (\cite{HKLR}). In addition, in our case all the fibers are algebraic.
 
\hfill
 
\begin{theorem}
\label{yyy}
Consider $(T^*V)\trice C$ as a complex analytic variety equipped with some complex structure $J\in\H$. Assume that $J\ne \pm I$. Then $(T^*V)\trice C$ admits a structure of an affine algebraic variety. Moreover, with these assumptions on $J$ the isomorphism class of $((T^*V)\trice C)_J$ does not depend on the choice of $J$.
\end{theorem}

\hfill

The paper is organised as follows. In Section 2 we review the general properties of K\"ahler, hyperk\"ahler and good quotients and relations between them. The only results in this section which are specifically relevant to our situation are \ref{_hol_symp_ham_}, \ref{_admits_good_quotient_} and \ref{_description_of_unstable_}. They seem to be not explicitly written anywhere. In Section 3 we give a precise formulation of \ref{mmm} and prove it. Section 4 deals with the proof of \ref{yyy}. We also study the moment map $\psi\colon (T^*V)\trice C\to\R_{\ge 0}$ for the $U(1)$-action. We prove that the function $-\psi$ is proper and strictly plurisubharmonic. Thus, we obtain a weaker statement that this hyperk\"ahler quotient is a Stein variety (\cite{_Fornaess_Narasimhan_}). Further, we show that if the Feix--Kaledin metric is complete then the tangent bundle of $V\twice C$ is big and nef. We do this using our previous result on the properness of $\psi$ and an unpublished result due to R.Bielawski. In Appendix we rewrite the Feix construction of the Feix--Kaledin hyperk\"ahler metric in a more canonical language. That allows us to see its relations to twisted cotangent bundles. 

\hfill

{\bf Acknowledgements:} I would like to thank my adviser Misha Verbitsky for suggesting me the problem. I am grateful to Dmitry Kaledin, Roger Bielawski and anonymous referees for reading a draft of this paper and their remarks on its content. I also thank Enrico Arbarello and Ignasi Mundet i Riera for useful remarks on certain parts of the proof. My thanks also go to my friend Renat Abugaliev for his interest and encouraging discussions. I was partially supported by the HSE University Basic Research Program, Russian Academic Excellence Project '5-100' as well as by Independent University of Moscow and Instituto de Matem\'atica Pura e Aplicada.

\section{Preliminaries on Quotients}
\subsection{K\"ahler and hyperk\"ahler reduction}

%%%%%%%%%%%%%%%%%%
%List of corrections in v.1.1.
% 1. Removed the connectedness assumption on C. Added in Definition 2.1. that two equivariance conditions are equivalent only when C is connected
% 2. Added regularity assumption to the hypotheses of Theorem 2.4. Stated a more general version of the theorem including the case when the quotient is an orbifold.
% 3. Mentioned that I choose to work with K\"ahler quotients only with respect to regular values of the moment map
% 4. Added a reminder on orbifolds

%%%%%%% LIST OF CHANGES IN V.1.2%%%%%
% 1. Removed the reminder on orbifolds.
% 2. Returned Theorem 2.4. to v.1.0. (no orbifolds) with a small change: I added that free action of C on \mu^{-1}(0) implies that 0 is a regular value.
% 3. Removed the passage in which I said that I choose to work with only regular values.
%%%%%

Let $X$ be a K\"ahler manifold, $\omega\in \Lambda^{1,1}(X)$ a K\"ahler form and $C$ a compact group acting on $X$. Suppose that the action of $C$ preserves the K\"ahler structure. Denote the Lie algebra of $C$ by $\c$. The differential of the $C$-action on $X$ induces a morphism of Lie algebras $L\colon \c\to \Gamma(TX)$ defined as usual by the formula
\begin{equation}
    L(\xi)_x = \frac{d}{dt}\exp(t\xi)x\restrict{t=0}
\end{equation}
for every $x\in M$. For convenience we shall slightly abuse the notation by denoting a vector $\xi \in \c$ and its image $L(\xi)$ in vector fields on $X$ by the same symbol $\xi$.

\hfill

\begin{definition}
\label{_moment_map_def_}
The action of $C$ on a K\"ahler manifold $X$ is called {\em Hamiltonian} if there exists a map $\mu\colon X \to \c^*$ satisfying the following properties:

\begin{itemize}
\item[(i)] For every $\xi\in \c$
\begin{equation}
d\<\mu,\xi\> = \xi\cntrct\omega \in \Omega^1(X,\R)
\end{equation}

\item[(ii)] (C-equivariance) For every $g\in C$
\begin{equation}
\mu(gx) = Ad^*g\cdot\mu(x)
\end{equation}
If $C$ is connected this is equivalent to saying that for every $\xi,\eta\in\c$
\begin{equation}
\omega(\xi, \eta) = -\<\mu,[\xi,\eta]\>
\end{equation}
\end{itemize}

In that case $\mu$ is called a {\em moment map}.
\end{definition}

\hfill

\remark
The definition does not make use of the complex structure on $X$ and can be given for every symplectic manifold $X$ on which a Lie group $C$ acts by symplectomorphisms. However, we shall restrict our attention only to the K\"ahler and later to the hyperk\"ahler case. 

\hfill

\remark A moment map is defined uniquely up to an addition of an $Ad^*$-invariant element of $\c^*$.

\hfill

One can introduce a similar notion in the case when $X$ is a hyperk\"ahler manifold and $C$ acts by preserving the hyperk\"ahler structure. We recall here the definition of a hyperk\"ahler manifold.

\hfill

\begin{definition}
\label{hk}
Suppose $(X,g)$ is a Riemannian manifold equipped with an algebra map $\H\to \End(TX)$ where $\H$ is the quaternion algebra. Suppose also that every complex structure $J\in \H \subset \End(TX)$ is integrable and orthogonal with respect to the metric $g$. For every complex structure $J\in \H$ define
\begin{equation}
    \omega_J(v,u) := g(Jv,u)
\end{equation}
If $\omega_J$ is closed for every $J$ then the manifold $X$ is called {\em hyperk\"ahler}.
\end{definition}

\hfill

For every hyperk\"ahler manifold $X$ we fix a triple of complex structures $(I,J,K)\in \End(TX)$ satisfying quaternionic relations $IJ = -JI = K$.

\hfill

\begin{definition} Let $C$ be a compact connected Lie group acting on a hyperk\"ahler manifold $X$. Suppose that the action of $C$ preserves the hyperk\"ahler structure. The action is called {\em Hamiltonian} if there exists a map $\mu\colon X \to \c^*\otimes \R^3$ 
\begin{equation}
    \mu = (\mu_I,\mu_J,\mu_K)
\end{equation}
such that $\mu_I$, $\mu_J$ and $\mu_K$ are moment maps with respect to K\"ahler forms $\omega_I$, $\omega_J$ and $\omega_K$ respectively (in the sense of \ref{_moment_map_def_}). The map $\mu\colon X\to\c^*\otimes\R^3$ is called a {\em hyperk\"ahler moment map}.
\end{definition}

\hfill

\begin{theorem}(\cite{HKLR}) 
\label{_red_main_theorem_}
Let $X$ be a K\"ahler (resp. hyperk\"ahler) manifold equipped with a Hamiltonian action of a compact Lie group $C$ with a moment map $\mu\colon X\to\c^*$ (resp. $\mu\colon X\to\c^*\otimes\R^3$). Consider the subspace $Z:=\mu^{-1}(0)\subset X$. Then the following holds
\begin{itemize}
    \item [(i)] The subspace $Z$ is $C$-invariant.
    \item [(ii)]Suppose $C$ acts on $Z$ freely. Then zero is a regular value of $\mu$, hence $Z$ is a submanifold of $X$. The quotient $Y := \mu^{-1}(0)/C$ is equipped with a natural structure of a K\"ahler (resp. hyperk\"ahler) manifold. This K\"ahler (resp. hyperk\"ahler) structure is the unique one such that for every $x\in Z$ projecting to a point $y\in Y$ the natural projection map from $\im(\c)^\perp \subset T_x Z$ to $T_y Y$ is an isomorphism of Hermitian (resp. hyper-Hermitian) vector spaces.
\end{itemize}
\end{theorem}

\hfill

\begin{definition} The manifold $Y$ from \ref{_red_main_theorem_} is called a {\em K\"ahler (resp. hyperk\"ahler) quotient} of $X$. In the K\"ahler case it is denoted by $X\twice C$ and in the hyperk\"ahler case by $X\trice C$.
\end{definition}

%We shall always assume that zero is a regular value of $\mu\colon X\to \c^*$ when talking about K\"ahler quotients. However, we shall sometimes need to consider hyperk\"ahler quotients when zero is {\em not} a regular value of $\mu\colon X\to \c^*\otimes\R^3$. The details are briefly explained in the next subsection.

%%%%%%%%%%%%%%%%%%%%%%%%%%%%%%%%%%%%%%%%%%%%%%%%%%%%%%%%%%%%
%%%%%%%%%%%%%%%%%%%%%%%%%%%%%%%%%%%%%%%%%%%%%%%%%%%%%%%%%%%%%%

\hfill

\subsection{Singular quotients}
\label{singular_quotients}

%%%%% minor changes
%%%%% LIST OF CHANGES IN V.1.2 %%%%%%
% 1. Emphasized in the statement of Theorem 2.6. that every K\"ahler stratum is a quotient with respect to a regular value.
%%%%%

In this subsection we state some more or less technical definitions and theorems to be used later. The reader is advised to skip it until needed.

%\begin{definition}
%Let $X$ be an orbifold. The {\em total space $T^*X$ of its cotangent bundle} is an orbifold defined locally as follows: if $U_x = V_x/H_x$ is as above then $T^*U_x := (T^*V_x)/H_x$.
%\end{definition}

Let $C$ be a compact Lie group acting on a manifold $X$ and $D\subset C$ be its subgroup. For a point $x\in X$ let $\operatorname{Stab}(x)$ be the stabilizer of $x$ in $C$. Following \cite{Sjamaar_Lerman} we introduce the following notation
\begin{gather*}
X^D: = \{x\in X|\: D\subset \operatorname{Stab}(x) \}\\
X_D := \{x\in X|\: D = \operatorname{Stab}(x) \}\\
X_{(D)} := \{x\in X|\: \operatorname{Stab}(x)\text{ is conjugate  to }D \}
\end{gather*}

As $C$ is compact the subset $X^D$ is a submanifold of $X$ containing $X_D$ as a dense open subset. Consider the normalizer $N(D)$ of the subgroup $D$ in $C$ and denote $W(D):= N(D)/D$. Then the group $W(D)$ acts properly and freely on $X_D$ and $X_{(D)}/C = X_D/W(D)$ is a disjoint union of smooth submanifolds.

Now let us return to the case when $X$ is  K\"ahler and the action of $C$ is Hamiltonian with a moment map $\mu\colon X\to\c^*$. 

\hfill

\begin{theorem}
\label{Sjamaar_Lerman}
(\cite{Sjamaar_Lerman}) The subsets $X_D$ are locally closed K\"ahler submanifolds of $M$. Moreover, the restriction of the moment map $\mu\colon X\to \c^*$ to $X_D$ induces the moment map $\mu_D\colon X_D\to \mathfrak w_D^*$ where $\mathfrak w_D$ is the Lie algebra of $W(D)$. Zero is a regular value of $\mu_D$ and $W(D)$ acts freely on $X_D\cap Z$. As a result the quotient $Y$ of $Z = \mu^{-1}(0)$ by $C$ decomposes as a union of K\"ahler manifolds $Y_{(D)} := (X_{(D)}\cap Z)/C = (X_D\cap Z)/W(D)$.
\end{theorem}

\hfill

\remark In fact Sjamaar and Lerman in \cite{Sjamaar_Lerman} prove a stronger statement, in particular, they show that the decomposition $Y = \bigcup Y_{(D)}$ is a stratification.

\hfill

A similar theorem holds in the hyperk\"ahler case.

\hfill

\begin{theorem}
\label{Dancer_Swann}
(\cite{Kaledin_thesis},\cite{Dancer_Swann}) Let $X$ be a hyperk\"ahler manifold with a Hamiltonian action of a group $C$. Let $\mu\colon X\to\c^*\otimes\R^3$ be a moment map. Then the hyperk\"ahler quotient $Y = X\trice C$ decomposes as a union of hyperkähler manifolds $Y_{(D)} := (X_{(D)}\cap Z)/C$.
\end{theorem}

\hfill

\remark
The decomposition of the hyperk\"ahler quotient $X\trice C$ as a union of hyperk\"ahler manifolds endows it with a structure of a stratified hyperk\"ahler space. We refer to \cite{Mayrand_stratification} for definitions and proofs.

\hfill

%%%%%%%%%%%%%%%%%%%%%%%%%%

\subsection{Algebraic and analytic quotients}

%%%%%%% List of corrections in v.1.1
% 1. Added Definition 2.12 and Proposition 2.13. The definition introduces the notion of a geometric quotient, and the proposition gives an example we are going to work with a lot. 
% 2. Added Theorem 2.14 which states Luna slice theorem in the needed generality. It used to prove the main result of the section, Proposition 2.19.
% 3. Added Definition 2.18 where I define the total space of the cotangent to an orbifold. It is needed to formulate Proposition 2.19.
% 4. Added an additional hypothesis on stabilizers to the statement of Proposition 2.19. The proof is completely rewritten with the use of Luna slice theorem.
% 5. Formulated Proposition 2.20 about forms on quotients under the assumption on stabilizers. The base variety is allowed to be non-smooth (an orbifold) in the new formulation of the proposition. A brief proof was given
% 6. Removed a theorem by Greb. Its formulation was not correct.
% 6. Some minor corrections

%%%%%%% LIST OF CHANGES IN V1.2 %%%%%%
% 1. Removed Luna slice theorem
% 2. Removed everything about orbifolds
% 3. Stated Proposition 2.16 (former Prop 2.19) for actions with trivial stabilizers
% 4. Formulated prop 2.17 (former Prop.2.20) for actions with trivial stabilizers. Added the statement about vector fields. Removed its proof.
%%%%

The procedure of taking quotients by non compact groups may sometimes be a subtle question. Na\"{i}ve topological quotients might look too ugly to work with. There exists a definition of a good quotient which mimics some standard properties of topological quotients and works well in the category of algebraic varieties over $\C$ or the category of complex analytic varieties. Either of the two mentioned categories will be denoted by $\mathcal C$.

\hfill

\begin{definition} (\cite{Hoskins_lecture_notes}, {\sf Section 3.5}; \cite{_Greb_}, {\sf Section 2.2 \& 2.3})
\label{_good_quot_def_}
Let $X$ be an object of $\mathcal C$ equipped with an action of a complex group $G \in Ob(\mathcal C)$. We call a $G$-invariant surjective map $p\colon X\to Y$ a {\em good quotient} map if it satisfies the following properties:

\begin{itemize}
\item[(i)] The map $p$ is affine (resp. locally Stein) i.e. the preimage of every affine (resp. Stein) open subset of $Y$ is affine (resp. Stein).
\item[(ii)] The natural map $\calO_Y \to (p_*\calO_X)^G$ is an isomorphism of sheaves.
\item[(iii)] $p$ maps $G$-invariant closed subvarieties to closed subvarieties. 
\item[(iv)] $p$ maps disjoint $G$-invariant closed subvarieties to disjoint subsets.
\end{itemize}
\end{definition}

\hfill

\remark Good quotients are often called {\em Hilbert quotients} especially in the analytic context.

\hfill

\remark Good quotients are always categorical quotients in the following sense. Every $G$-invariant map from $X$ to some $Y'\in Ob(\mathcal C)$ can be uniquely factorized through the projection $p$ from $X$ to the good quotient $Y$ (\cite{Hoskins_lecture_notes}, Prop. 3.30).

\hfill

We shall denote the good quotient of $X$ by $G$ by the symbol $X/G$. This is not the notation commonly used in the literature but it seems that it will enable us to avoid confusion with K\"ahler quotients.

The only thing we shall need to know about the existence of good quotients is the following.

\hfill

\begin{proposition} (\cite{Mumford_Fogarty_Kirwan}, {\sf Ch. 1, \S 2, Thm. 1.1}; \cite{Hoskins_lecture_notes}, {\sf Thm. 4.30}; \cite{Greb_Miebach_Stein})
\label{_existence_of_quotient_affine_}
Let $X\in Ob(\mathcal C)$ be a complex variety equipped with an action of a complex reductive group $G$.
\begin{itemize}
    \item[(i)] Suppose that $X$ is affine (resp. Stein). Then $X$ admits a good quotient by $G$.
    \item[(ii)] Suppose that $X$ admits an affine (resp. locally Stein) $G$-equivariant map to a $G$-variety $M$. Suppose that $M$ admits a good quotient $N$. Then $X$ admits a good quotient $Y$ and $Y$ is affine (resp. locally Stein) over $N$.
\end{itemize}
\end{proposition}

\hfill

\remark The proof in the algebraic case would follow from the following fact. If $X = \mathrm{Spec}(A)$ where $A$ a $\C$-algebra of finite type then the good quotient $X/G$ is isomorphic to $\mathrm{Spec}(A^G)$. The second assertion follows from the first as good quotients are always affine.

\hfill

The next theorem relates quotients in algebraic and analytic categories.

\hfill

\begin{theorem} (\cite{_Luna_},\cite{Heinzner_Huckleberry}) Let $X$ be an algebraic variety over $\C$ and $G$ a reductive group acting on $X$ algebraically. Assume that there exists a good quotient map $p\colon X\to Y$ to an algebraic variety $Y$. Then $p\colon X^{an}\to Y^{an}$ is a good quotient map in the category of complex analytic varieties.
\end{theorem}

\hfill

\begin{definition}
A good quotient $p\colon X\to X/G=:Y$ is called a {\em geometric quotient} if the preimage of any point of $Y$ is a closed orbit in $X$.
\end{definition}

\hfill

Equivalently, a good quotient is geometric if every $G$-orbit in $X$ is closed (\cite{Hoskins_lecture_notes}, {\sf Cor. 3.32}). Though not all good quotients are geometric, there is an important case when they are.

\hfill

\begin{proposition}
\label{_geometric_quot_}
Let $G$ be a reductive group acting on $X$ with finite stabilizers. If $X$ admits a good quotient by $G$ then this quotient is geometric.
\end{proposition}
\begin{proof}
The assumption on stabilizers guarantees that all orbits have the same dimension equal to $\dim G$. If there was a non-closed orbit $G\cdot x$ then its closure $\overline{G\cdot x}$ would contain an orbit of smaller dimension.
\end{proof}

\hfill

We finish the subsection with an example of a quotient which will be pivotal in the sequel. The motivation for this example comes from the procedure of a holomorphic symplectic reduction which we shall now describe. Let $X$ be a holomorphic symplectic variety with a holomorphic symplectic form $\Omega\in \Lambda^{2,0}X$. Suppose $X$ is equipped with an action of a complex reductive group $G$ preserving holomorphic symplectic structure. The definition below mimics \ref{_moment_map_def_}.

\hfill

\begin{definition}
\label{hol_symp_moment_map}
The action of $G$ on a holomorphic symplectic variety $X$ is called {\em Hamiltonian} if there exists a map $\calM\colon X\to \g^*$ with the following properties:

\begin{itemize}
\item[(i)]For every $\xi\in \g$
\begin{equation}
d\<\calM,\xi\> = \xi\cntrct\Omega \in \Lambda^{1,0}(X)
\end{equation}

\item[(ii)] (G-equivariance) For every $g\in G$
\begin{equation}
\calM(gx) = Ad^*g\cdot\calM(x)
\end{equation}
\end{itemize}

The map $\calM\colon X\to \g^*$ is called a {\em holomorphic symplectic moment map}.
\end{definition}

\hfill

\remark The moment map $\calM$ is automatically holomorphic.

\hfill

Similarly to \ref{_red_main_theorem_} one can consider a $G$-invariant subvariety $Z := \calM^{-1}(0)\subset X$. If $Z$ admits a good quotient $Y$ then $Y$ inherits the holomorphic symplectic structure from $X$. 

\hfill

\begin{definition}
\label{hol_symp_quot_def}
Suppose $Z$ admits a good quotient. Then we call $Z/G$ a {\em holomorphic symplectic quotient} of $X$ and denote it by $X\twice G$.
\end{definition}

\hfill

Now let $M\in Ob(\mathcal C)$ be a smooth variety and $X$ the total space of its cotangent bundle. Let $G\in Ob(\mathcal C)$ be a complex reductive group acting on $M$. The action of $G$ can be naturally extended to $X$ as
\begin{equation}
\label{g_action_on_cotangent}
g\cdot(x,\alpha) := (gx, (g^{-1})^*\alpha)
\end{equation}
where $x\in M$, $\alpha \in T^*_xM$. The inclusion $\iota$ of $M$ as the zero section and the projection $\pi$ from $X$ to $M$ become $G$-equivariant maps with respect to this action.

The manifold $X$ is equipped with the standard holomorphic symplectic $2$-form $\Omega\in \Lambda^{2,0}(X)$. More precisely, 
\begin{equation}
\label{standard_two_form}
    \Omega = -d\tau
\end{equation}
where $\tau$ is the tautological holomorphic $1$-form on $X$. Recall that $\tau$ is defined as follows
\begin{equation}
\label{tautological_one_form}
    \tau_{(x,\alpha)}(v) := \alpha(\pi_*v)
\end{equation}
for every $v\in T_{(x,\alpha)}X$.

\hfill

\begin{proposition}
\label{_hol_symp_ham_}
The action of $G$ on $X$ preserves $\tau$ and hence $\Omega$. It is Hamiltonian with a moment map $\calM\colon X\to \g^*$ (\ref{hol_symp_moment_map}) 
\begin{equation}
\label{_moment_map_cot_}
\<\calM(x,\alpha),\xi\> = \<\alpha,\xi\>
\end{equation}
where on the right-hand side we consider $\xi$ to be an element of $T_xM$.
\end{proposition}

\begin{proof}
By definition 
\begin{equation}
\tau_{(x,\alpha)}(v) = \alpha(\pi_*v)
\end{equation}
Hence for every $g\in G$
\begin{gather*}
g^*\tau_{(x,\alpha)}(v) = \tau_{(gx,(g^{-1})^*\alpha)}(g_*v) =\\
= (g^{-1})^*\alpha(\pi_*g_*v) = (g^{-1})^*\alpha(g_*\pi_*v) = \\
=\alpha (\pi_*v) = \tau_{(x,\alpha)}(v)
\end{gather*}
and the $G$-action preserves $\tau$. We obtain that for every $\xi\in \g$
\begin{equation}
0 = \mathsf L_\xi\tau = d(\xi\cntrct\tau) - \xi\cntrct\Omega
\end{equation}
and 
\begin{equation}
\<\calM(x,\alpha),\xi\> := \xi\cntrct\tau = \<\alpha, \xi\>
\end{equation}
is a moment map. The $G$-equivariance of $\calM$ is easily checked.
\end{proof}

\hfill

\begin{proposition}
\label{_admits_good_quotient_}
Suppose that $G$ acts on a smooth variety $M$ with trivial stabilizers and $M$ admits a good (and therefore geometric) quotient $p\colon M\to M/G =: N$. Then:
\begin{itemize}
    \item [(i)] The variety $N$ is smooth.
    \item [(ii)] The subvariety $Z:=\calM^{-1}(0) \subset X:=T^*M$ is smooth and admits a smooth geometric quotient $Y:= Z/G$.
    \item [(iii)] The quotient $Y$ is naturally isomorphic to $T^*N$.
\end{itemize}
\end{proposition}
\begin{proof} \begin{itemize}
    \item [(i)] Let $x$ be a point of $M$. Choose any smooth complex analytic subvariety $U$ in a neighbourhood of $x$, transversal to $G\cdot x$ and having complementary dimension to this orbit. By choosing $U$ sufficiently small we can guarantee that $U$ is mapped isomorphically onto a neighbourhood of $p(x)$.
    \item [(ii)] For smoothness of $\calM^{-1}(0)$ it is enough to prove that the differential $d\calM\colon T_xX\to \g^*$ is surjective for any point $x\in \calM^{-1}(0)$. We describe the kernel of $d\calM$ to show this.
    \begin{gather*}
    \ker d\calM = \{v\in T_xX\:|\: v \cntrct d\<\calM,\xi\> = 0 \:\forall \xi\in\g\} = \\
    =\{v\in T_xX\:|\: \Omega(\xi,v) = 0 \:\forall \xi\in\g\} = (\im (\g\to T_xX))^{\perp_\Omega}
    \end{gather*}
    As $G$ acts on $M$ with trivial stabilizers, it also must act with trivial stabilizers on $X$. Hence the map $\g\to T_xX$ is injective for any $x\in X$. The dimension count shows $\dim_\C(\ker d\calM) = \dim_\C X - \dim_\C G$. The differential $d\cal M$ is therefore surjective.
    
    The projection map $\pi\colon Z\to M$ is clearly affine (if $M$ is assumed to be an algebraic variety) or locally Stein (if $M$ is assumed to be a complex analytic variety). \ref{_existence_of_quotient_affine_} guarantees that $Z$ admits a good quotient $Y:=Z/G$ as soon as $M$ does. By \ref{_geometric_quot_} the quotient $Y$ is geometric. The proof of the first part of the proposition applied to the quotient $Y = Z/G$ shows that $Y$ is smooth.
    \item[(iii)] The map $f\colon T^*N\to Y$ is defined as follows. Let $y$ be a point of $N$. Choose any $x\in p^{-1}(y)\subset M$. The pullback map $p^*\colon T^*_yN\to T^*_xM$ maps $T^*_yN$ isomorphically onto the space of elements of $T^*_xM$ vanishing on the tangent space in $x$ to the orbit $G\cdot x$. \ref{_hol_symp_ham_} implies that for every $\alpha\in T^*_yN$ the value of $\calM$ at $p^*\alpha$ vanishes. We define $f(\alpha)\in Y=\calM^{-1}(0)/G$ to be the image of $p^*\alpha\in\calM^{-1}(0)$ in $Y$. One can check that the map is defined correctly and is an isomorphism.
    
\end{itemize}

\end{proof}

\hfill

The following proposition is a classical fact about submersions. We shall omit its proof.

\hfill

\begin{proposition}
\label{_forms_on_quot_}
Suppose that $G$ acts on a smooth variety $M$ with trivial stabilizers and $p\colon M\to N$ is the quotient map. For every $r\in\mathbb Z_{> 0}$ let $\Omega_{hor}^{r}\subset \Omega^r M$ be the subsheaf of horizontal $r$-forms i.e. those $r$-forms $\alpha \in \Omega^r M$ such that $\xi\cntrct \alpha = 0$ for every $\xi\in \g$. Then the pullback map on $r$-forms
\begin{equation}
\label{fck}
    \Omega^r N \to (p_*\Omega_{hor}^{r})^G
\end{equation} 
and the pushforward map on vector fields
\begin{equation}
    (p_*(\mathcal TM/\g\cdot\calO_M))^G \to \mathcal TN
\end{equation}
are both isomorphisms of sheaves of $\calO_N$-modules.
\end{proposition}

\hfill

%%%%%%%%%%%%%%%%%%%%%%%%%%%%
%%%%%%%%%%%%%%%%%%%%

\subsection{K\"ahler quotients as good quotients} 

%%%%%% List of corrections %%%%%%%
% 1. Added to Definition 2.21 the definition of stable points. Added to Prop. 2.22, Prop.2.25, Theorem 2.28 and Cor. 2.31 the corresponding statement for stable points. Added Cor. 2.26 as a clarification of the picture in the case of regular values of moment maps.
% 2. Removed from the remark after Prop. 2.22 the claim that the set of polystable points is dense. This claim is false.
% 3. Added a reference to Mundet i Riera 2000
% 4. Some minor changes

Our goal now is to present a K\"ahler quotient of a K\"ahler manifold $X$ as a geometric or at least good quotient of some open subset of $X$. An interested reader can consult a survey \cite{GRS} for the compact case and a survey \cite{Heinzner_Huckleberry} for the non-compact case. 

\hfill

\begin{definition}
\label{stability_def}
Let $G$ be a complex reductive group acting holomorphically on a K\"ahler manifold $X$. Suppose that the restriction of this action to a fixed maximal compact subgroup $C$ of $G$ preserves the K\"ahler structure and is Hamiltonian with a moment map $\mu\colon X\to \c^*$. 
\begin{itemize}
\item[(i)] A point $x\in X$ is called {\em semistable} if 
\begin{equation}
\overline{Gx}\cap \mu^{-1}(0) \ne \varnothing
\end{equation}
We denote the set of semistable points of $X$ by $X^{ss}$. We shall call a point {\em unstable} if it is not semistable. We shall denote the set of unstable points by $X^{us}$.

\item[(ii)] A point $x\in X$ is called {\em polystable} if 
\begin{equation}
Gx\cap \mu^{-1}(0)\ne \varnothing
\end{equation}
We denote the set of polystable points as $X^{ps}$.

\item [(iii)] A point $x\in X$ is called {\em stable} if it is polystable and the stabilizer of $x$ in $G$ is finite.
We denote the set of stable points as $X^{s}$

\end{itemize}
\end{definition}

\hfill

It follows easily from the definitions that the sets $X^{ss}$, $X^{us}$, $X^{ps}$ and $X^s$ are $G$-invariant.

\hfill

\begin{proposition} (\cite{Heinzner_Huckleberry},\cite{GRS}, {\sf Thm. 7.2}) The subsets $X^s$ and $X^{ss}$ are open in $X$.
\end{proposition}

\hfill

\remark
The set $X^{ps}$ of polystable points is not necessarily open. 

\hfill

Now we want to construct a map from $X^{ss}$ to $\mu^{-1}(0)/C$ i.e. to the K\"ahler quotient of $X$. To do this we need the following two results.

\hfill

\begin{lemma} Suppose $x,y\in\mu^{-1}(0)$ lie in the same $G$-orbit. Then they lie in the same $C$-orbit.
\end{lemma}

\begin{proof} Let $y = gx$. By the polar decomposition $g = g_0\exp(\sqrt{-1}\xi)$ for some $g_0\in C$ and $\xi\in\c$. Consider the curve $y(t) = \exp(\sqrt{-1}t\xi)g_0 x$ connecting $g_0x$ and $y$, then
\begin{equation}
\frac{d}{dt}\<\mu(y_t),\xi\>\restrict{t=\tau} = \omega(\xi_{y(\tau)},I\xi_{y(\tau)}) = ||\xi_{y(\tau)}||^2\ge 0
\end{equation}
As $\mu(g_0 x) = \mu(y) = 0$ we see that $\xi_{g_0 x} = 0$ and hence $y = g_0 x$.
\end{proof}

\hfill

\begin{lemma} (\cite{Heinzner_Huckleberry}, {\sf Cor. 4.2.2}) Every semistable orbit contains a unique polystable orbit in its closure.
\end{lemma}

\hfill

Now we can define a map $\Psi\colon X^{ss} \to \mu^{-1}(0)/C$ as follows. For a semistable point $x\in X$ consider the unique polystable orbit in the closure of its $G$-orbit. It intersects $\mu^{-1}(0)$ in a $C$-orbit $\tilde x\in \mu^{-1}(0)/C$. We define
\begin{equation}
\label{psi}
\Psi(x) := \tilde{x}
\end{equation}

\hfill

\begin{proposition}
\label{ss}
(\cite{Heinzner_Huckleberry}, {\sf Thm. 4.2.4}) Let $G$ be a complex reductive group acting holomorphically on a K\"ahler manifold $X$. Suppose that the restriction of this action to a fixed maximal compact subgroup $C$ of $G$ preserves the K\"ahler structure and is Hamiltonian with a moment map $\mu\colon X\to \c^*$. 
\begin{itemize}
    \item [(i)] The map $\Psi\colon X^{ss} \to \mu^{-1}(0)/C$ is a good quotient map. 
    \item [(ii)] Let $\mu^{-1}(0)^s$ be the set of points in $\mu^{-1}(0)$ whose stabilizer in $C$ is finite. Then the map $\Psi\restrict{X^s}\colon X^s\to \mu^{-1}(0)^s/C$ is a geometric quotient.
\end{itemize}
\end{proposition}

\hfill

\begin{corollary}
Assume that zero is a regular value of the moment map. Then all the inclusions $X^s\subset X^{ps}\subset X^{ss}$ are equalities. Moreover, the map $\Psi\colon X^s\to \mu^{-1}(0)/C$ is a geometric quotient.
\end{corollary}

\begin{proof}
If $x\in X$ be a polystable point, then the orbit $G\cdot x$ intersects $\mu^{-1}(0)$. Every point of $\mu^{-1}(0)$ has finite stabilizer. Hence the stabilizer of $x$ is also finite and $x$ is stable. If $x\in X$ is semistable but not polystable then $\overline{G\cdot x}$ contains a polystable orbit $G\cdot y$ such that $\dim G\cdot y < \dim G\cdot x$. But as we have seen earlier the point $y$ must have finite stabilizer, hence $\dim G\cdot y$ is maximal, contradiction.
\end{proof}

It is usually non-trivial to say if a given point is (semi)stable or not. One of the ways which can be practically used to answer this question is the Hilbert-Mumford criterion which we shall state after giving the necessary definition.

\hfill

\begin{definition}
\label{mu_weight_def}
The {\em $\mu$-weight} of a pair $(x,\xi)\in X\times\c$ is a number $w^\xi_\mu(x)\in\R\cup\{\pm\infty\}$ defined by the equation
\begin{equation}
\label{_mu_weight_}
w^\xi_\mu(x) := \lim\limits_{t\to\infty} \<\mu(\exp(\sqrt{-1}t\xi)x),\xi\>
\end{equation}
\end{definition}

\hfill

\remark One can prove that the limit in formula~(\ref{_mu_weight_}) is always well-defined. Indeed,  the function $\psi(t):=\<\mu(\exp(\sqrt{-1}t\xi)x),\xi\>$ is non-decreasing as
$$
\frac{d}{dt}\psi = \omega(\xi,I\xi) = ||\xi||^2\ge 0
$$
by the very definition of a moment map. 

\hfill

\begin{theorem}
\label{Hilbert_Mumford}
({\sf "Hilbert-Mumford criterion"}, \cite{Mundet_I_Riera_st},\cite{_Teleman_}) \begin{itemize}
    \item [(i)] A point $x\in X$ is stable if and only if $w^\xi_\mu(x)> 0$ for every $\xi\in \c\setminus\{0\}$.
    \item [(ii)] Assume that the action of $G$ on $X$ is energy complete (\cite{_Teleman_}, {\sf Def. 2.8}). A point $x\in X$ is semistable if and only if $w^\xi_\mu(x)\ge 0$ for every $\xi\in \c\setminus\{0\}$.
\end{itemize}
\end{theorem}

\hfill

We shall not need the definition of energy complete actions. By \cite{_Teleman_}, {\sf Prop. 2.9} any linear action is energy complete, and this is the only case to which we shall apply the criterion.

\hfill

\remark A similar description of polystable points was given by I. Mundet i Riera in \cite{Mundet_I_Riera}.

\hfill

Let us now apply the Hilbert-Mumford criterion to a specific example.

\hfill

\begin{example}
\label{Hilbert_Mumford_ex}
Consider a unitary representation of $C$ in an Hermitian vector space $(V,h)$. The complex manifold $V$ admits the standard K\"ahler form $\omega = \Im h$ which is preserved by the action of $C$. Moreover the action of $C$ is Hamiltonian and the moment map $\mu\colon V\to\c$ is given by the formula
\begin{equation}
\label{_moment_map_unit_rep_}
\<\mu(v),\xi\> = \frac{1}{2}\omega(\xi v,v) + \<\theta,\xi\>
\end{equation}
where $\theta$ is some $Ad^*$-invariant element of $\c^*$. We omit the derivation of this formula.

Let us introduce the notation. Every $\xi\in\c$ is a skew-symmetric operator on $V$ i.e. $\xi\in\mathfrak{u}(V)$. Hence its eigenvalues are imaginary. We denote them by $\sqrt{-1}\lambda_i$. Given a vector $v\in V$ we can write it as the sum $v = \sum\limits_i v_i$ where $v_i$ lies in the eigenspace corresponding to the eigenvalue $\sqrt{-1}\lambda_i$. Consider the set of indices $I_v:=\{i\:|\: v_i\ne 0\}$.

\hfill

\begin{proposition}
\label{fff}
The $\mu$-weight of a pair $(v,\xi)$ is given by the formula
\begin{equation}
\label{mu_weight}
w_\mu^\xi (v) = \begin{cases}
+\infty,&\text{if } \exists i\in I_v \text{ s.t. } \lambda_i<0\\
\<\theta,\xi\>, & \text{otherwise}
\end{cases}
\end{equation}
\end{proposition}

\begin{proof} Let $v = \sum\limits_i v_i$. Then $\exp(\sqrt{-1}t\xi)v = \sum \exp(-\lambda_i t) v_i$. Hence
\begin{gather*}
\<\mu(\exp(\sqrt{-1}t\xi)v,\xi\> = \frac{1}{2}\omega\left( I \sum\limits_i \lambda_i \exp(-\lambda_i t) v_i, \sum\limits_i v_i\right) + \<\theta,\xi\> = \\
= -\frac{1}{2}\sum_i \lambda_i \exp(-\lambda_i t) ||v_i||^2 + \<\theta,\xi\>
\end{gather*}
where the last identity follows from the fact that eigenspaces of a unitary operator are orthogonal. Taking the limit as $t\to \infty$ we obtain the result.
\end{proof}

\hfill

\begin{corollary}
\label{_description_of_unstable_}
The set of (semi)stable points of a vector space $V$ is a complement of a union of vector subspaces in $V$, more precisely
\begin{equation}
\label{_us_description_}
V^{s} =  V\setminus\left(\bigcup\limits_{\<\theta,\xi\> \le 0} V^{\xi\ge 0}\right);\:\:\:\:\:
V^{ss} = V\setminus\left(\bigcup\limits_{\<\theta,\xi\> < 0} V^{\xi\ge 0}\right)
\end{equation}
where the union is taken over $\xi\in\c$. Here $V^{\xi\ge 0}$ is the sum of eigenspaces of $\xi$ with eigeinvalues $\sqrt{-1}\lambda_i$ with non-negative $\lambda_i$. %Moreover, the union of vector subspaces is finite in both cases.
\end{corollary}

\begin{proof} By \ref{fff} the points in $\bigcup\limits_{\<\theta,\xi\> \le 0} V^{\xi\ge 0}$ are precisely the points with a non-positive $\mu$-weight and the points in $\bigcup\limits_{\<\theta,\xi\> < 0} V^{\xi\ge 0}$ are precisely the points with a negative $\mu$-weight. By the Hilbert-Mumford criterion (\ref{Hilbert_Mumford}) the point is stable if and only if its $\mu$-weight is positive and semistable if and only if its $\mu$-weight is non-negative. %By Proposition~\ref{ss} the analytically Zariski open subsets $V^s$ and $V^{ss}$ admit a good quotient. Hence by Theorem~\ref{Greb} they are algebraic. Therefore their complements $\bigcup\limits_{\<\theta,\xi\> \le 0} V^{\xi\ge 0}$ and $\bigcup\limits_{\<\theta,\xi\> < 0} V^{\xi\ge 0}$ respectively are finite unions of irreducible closed subvarieties.
\end{proof}
\end{example}

\hfill

\remark It is worth mentioning that an investigation of linear representations from the viewpoint of GIT was carried out by V. Hoskins in \cite{Hoskins_linear}. %However, it seems that Corollary~\ref{_description_of_unstable_} does not follow immediately from her results.

\hfill

%%%%%%%%%%%%%%%%%%%%%
%%%%%%%%%%%%%%%%%%%%%%

\section{Geometry of Total Spaces of Cotangent Bundles}

%%%%%%%%%%%%%%%%%%%%%%%%%%%%%%%%%%%%%%%%%
%%%%%%%%%%%%%%%%%%%%%%%%%%%%%%%%%%%%%%%%%

\subsection{Semistable points on cotangent bundles}

%%%%%%%% List of corrections in v.1.1. %%%%%%%
% 1. Removed the description of semistable points on T^*V from Prop. 3.2 as it was false. Rewrote the proof of Cor. 3.3 which still holds. It's Cor. 3.3 that I need later, not Prop. 3.2.
% 2. Added to the formulation of Prop. 3.2 and Cor. 3.3 the case of stable points.
%%%%%% no changes in v.1.2 %%%%%%%%

The most basic example of the Feix--Kaledin structure is the standard hyperk\"ahler structure on the total space $X$ of the cotangent bundle of an Hermitian vector space $V$. As a complex vector space $X_I = V\oplus V^*\cong V \oplus \overline{V}$ where the last isomorphism is induced by the Hermitian metric on $V$. As a real vector space $X_\R$ is just the direct sum of two copies of the vector space $V_\R$. The complex structures can be written as follows
\begin{gather}
    I(x,y) = (Ix,-Iy)\\
    J(x,y) = (-y,x)
\end{gather}
and the Riemannian metric is just the direct sum of the metrics on $V$. We slightly abuse the notation by denoting the complex structure $I$ on $V_\R$ and on $X_\R$ by the same symbol.

We shall now use the arguments and constructions of \ref{Hilbert_Mumford_ex} to give a description of semistable points on cotangent bundles to Hermitian vector spaces.

Consider a unitary representation of a compact connected Lie group $C$ in an Hermitian vector space $(V,h)$. It can be uniquely extended to the complex representation of the complexification $G$ of $C$. In formula~(\ref{g_action_on_cotangent}) we defined an action of the group $G$ on the total space of cotangent bundle to any $G$-manifold. In the same way we define the action of $G$ on the cotangent bundle $X$ to $V$. This action is linear on $X$ as a complex vector space $V\oplus V^*$. The representation of $G$ on $X$ is equal to the direct sum of the original representation and its dual. First, we shall consider a more general case of the direct sum of two arbitrary representations.

\hfill

\begin{lemma}
\label{_direct_sum_}
Let $U = V\oplus W$ be the direct sum of two representations of $C$. Then every moment map $\mu_V\colon V\to \c^*$ can be extended uniquely to a moment map $\mu_U\colon U\to \c^*$.
\end{lemma}

\begin{proof}
Formula~(\ref{_moment_map_unit_rep_}) implies that every moment map $\mu_U\colon U\to \c^*$ has the form
$$
\<\mu_U(x,y),\xi\> = \frac{1}{2}\left(\omega_V(\xi x,x) + \omega_W (\xi y, y)\right) + \<\theta, \xi\>
$$
for some $C$-invariant $\theta\in \c^*$. The moment map $\mu_U$ is uniquely defined by the covector $\theta\in \c^*$ which can be read from the restriction of $\mu_U$ to $V$.
\end{proof}

\hfill

\begin{proposition}
\label{_direct_sum_ss_}
Let $U =V\oplus W$ be the direct sum of two representations of $C$. Denote by $p\colon U\to V$ the projection to the first factor. Fix a moment map $\mu_V\colon V\to \c^*$. Let $\mu_U\colon U\to \c^*$ be the moment map on $U$ extending $\mu_V$. (Its existence and uniqueness is guaranteed by \ref{_direct_sum_}.) Let $V^{s}$ (resp. $V^{ss}$) be the set of stable (resp. semistable) points of $V$ with respect to the moment map $\mu_V$, the analogous notation is used for $U$. Then
$$
p^{-1}(V^{s}) \subset U^{s};\:\:\:\:\:p^{-1}(V^{ss}) \subset U^{ss}
$$
\end{proposition}

\begin{proof} By applying \ref{_description_of_unstable_} to $U = V\oplus W$ we obtain that
$$
U\setminus U^{s} = \bigcup\limits_{\<\theta,\xi\><0} \left(V\oplus W\right)^{\xi\ge 0} = \bigcup\limits_{\<\theta,\xi\><0} \left(V^{\xi\ge 0}\oplus W^{\xi\ge 0}\right) \subset \bigcup\limits_{\<\theta,\xi\><0} \left(V^{\xi\ge 0}\oplus W\right) = p^{-1}(V\setminus V^{s})
$$
In other words, $p^{-1}(V^{s}) \subset U^{s}$. The proof for semistable points is similar.
\end{proof}

\hfill

\begin{corollary}
\label{_fine_}
Let $V$ be a unitary representation of a compact Lie group $C$ and $X$ the cotangent bundle to $V$. Then the preimage of $V^s$ (resp. $V^{ss}$) under the projection $\pi\colon X\to V$ is contained in $X^s$ (resp. $X^{ss}$).
\end{corollary}

\begin{proof}
Follows by applying \ref{_direct_sum_ss_} to $U = T^*V = V\oplus V^*$.
\end{proof}

\hfill

%%%%%%%%%%%%%%%%

\subsection{Quotients of affine spaces}

%%%%%%%% LIST OF CHANGES  in v.1.2 %%%%%%%
% 1. Made the preambule to Theorem 3.4 shorter and more clear.
% 2. Removed a modification of Theorem 3.4 as it was not needed and probably false in the way it was formulated.
% 3. Completely rewrote the first step of the proof of Theorem 3.4.
% 4. Changed Step 2 in such a way that stable points with trivial stabilizer are considered instead of semistable ones. Changed notation according to that.
% 5. Minor changes
%%%%%%%%%
%%%%%%%% List of changes in v.2.0 %%%%%%
% 1. Stated theorem 3.4. in a less generality. Added the previous version of the theorem as theorem 3.8 in the end of the subsection. Stated also that the FK metric on the cotangent bundle to the quotient is the hyperk\"ahler quotient metric.

We consider again an Hermitian vector space $V$ with a unitary linear action of a compact group $C$. We already know that the action of $C$ on $V$ is Hamiltonian i.e. there exists a moment map $\mu\colon V\to\c^*$. By \ref{_direct_sum_} this moment map extends to the moment map $\mu_I\colon X\to \c^*$ where $X:=T^*V$. The moment map $\mu_I$ and the holomorphic symplectic moment map $\calM\colon X\to \g^* = \c^*\otimes\C$ as in \ref{_hol_symp_ham_} form together a hyperk\"ahler moment map
$$
\mu_{hk}:= (\mu_I,\Re\calM, \Im\calM)\colon X\to \c^*\otimes\R^3
$$

\hfill

\begin{theorem}
\label{very_long_proof}
Let $V$ be an Hermitian vector space. Suppose that a complex reductive group $G$ acts on $V$ linearly and the restriction of the action of $G$ to some fixed maximal compact subgroup $C\subset G$ preserves the Hermitian structure. Assume that the action of $C$ on $V$ is generically free i.e. the set $V_0$ of points with trivial stabilizer is open in $V$. Let $N := V\twice C$ be the K\"ahler quotient of $V$ by $C$ with respect to a fixed moment map $\mu\colon V\to\c^*$. Then the Feix--Kaledin metric is globally defined on the total space $Y_0$ of the cotangent bundle to the open K\"ahler stratum $N_0:=V_0\twice C$ of $N$ (see \ref{Sjamaar_Lerman} for the definition of the K\"ahler strata). Moreover, the manifold $Y_0$ is obtained as a hyperk\"ahler quotient of a certain open subset of $T^*V$. The hyperk\"ahler quotient metric on $Y_0$ coincides with the Feix--Kaledin metric.
\end{theorem}

\hfill

Before going to the proof of \ref{very_long_proof} we need to state several auxiliary results. The next two lemmas are concerned with descent of group actions to quotients.

\hfill

\begin{lemma}
\label{lemmaa}
Let $S\in Ob(\mathcal C)$ be a complex variety in the algebraic or analytic category equipped with an action of a reductive group $G$. Suppose also that a group $H$ acts on $S$ algebraically (resp. holomorphically) and the action of $H$ commutes with $G$. Assume that there exists the good quotient $T = S/G$. Then the action of $H$ on $S$ descends uniquely to an action of $H$ on $T$ in such a way that the quotient map is $H$-equivariant.
\end{lemma}

\begin{proof} We shall prove the lemma in the algebraic case, the proof in the complex analytic case is similar. As the good quotient map is by definition affine it is enough to prove the statement in the case when $S$ is affine i.e. $S=\mathrm{Spec}(A)$ where $A$ is a $\C$-algebra of finite type. In that case $T = \mathrm{Spec}(A^G)$ by \ref{_existence_of_quotient_affine_}. As the action of $H$ commutes with the action of $G$ the subalgebra $A^G$ is preserved by the $H$-action on $A$. This action on $A^G$ induces the action of $H$ on $T$.
\end{proof}

\hfill

\begin{lemma}
\label{lemma}
Let $M$ be a K\"ahler manifold and $X$ be the total space of the cotangent bundle to $M$. Suppose that $X$ admits the globally defined Feix--Kaledin hyperk\"ahler metric. Suppose that $M$ is equipped with a Hamiltonian action of a compact Lie group $C$. Consider the natural lifting of this action to $X$. Fix a hyperk\"ahler moment map $\mu_{hk}\colon X\to\c^*\otimes\R^3$ on $X$ of the form
$$
    \mu_{hk} = (\mu_I, \Re \mathcal M, \Im \mathcal M)
$$
with $\mathcal M$ as in \ref{_hol_symp_ham_} and $\mu_I$ an arbitrary moment map for $\omega_I$. Consider the $U(1)$-action on $X$ by fiberwise multiplication. Then $\mu^{-1}(0)$ is $U(1)$-invariant. Moreover, the action of $U(1)$ on $\mu^{-1}(0)$ descends to the action of $U(1)$ on the hyperk\"ahler quotient $X\trice C$ which is HKLR-compatible with the hyperk\"ahler metric.
\end{lemma}

\begin{proof} It is a classical fact. From the explicit description of $\mathcal M$ (\ref{_hol_symp_ham_}) we notice that $\mathcal M^{-1}(0)$ is $U(1)$-invariant. Hence the $U(1)$-invariance of $\mu_{hk}^{-1}(0)$ will follow from the $U(1)$-invariance of $\mu_I$.

For every $\lambda\in U(1)$ the function $\lambda^*\mu_I$ is still a moment map for the action of $C$. Indeed, for every $\xi\in\c$
$$
d\<\lambda^*\mu,\xi\> = \lambda^*\<d\mu,\xi\> = \lambda^*(\xi\cntrct\omega_I) = \xi\cntrct\omega_I
$$
as the vector field $\xi$ and the $2$-form $\omega_I$ are both $U(1)$-invariant. Hence the functions $\lambda^*\mu_I$ and $\mu_I$ differ by a constant. As they are certainly equal on $M$ they are also equal on $X$.

The $U(1)$-action commutes with the $C$-action, thereby descending to the action of $U(1)$ on the quotient $X\trice C$. Let $\tilde{\omega_I}$ and $\tilde{\Omega}$ denote the K\"ahler form and the holomorphic $2$-form on the quotient. They are uniquely defined by the properties that $\rho^*\tilde{\omega_I} = \omega_I\restrict{\mu^{-1}(0)}$ and $\rho^*\tilde{\Omega} = \Omega\restrict{\mu^{-1}(0)}$ where $\rho\colon \mu^{-1}(0)\to X\trice C$ is the quotient map. Hence the compatibility properties (\ref{compatible_action})
\begin{gather}
    \lambda^*\omega_I = \omega_I\\
    \lambda^*\Omega = \lambda\Omega
\end{gather}
are preserved by taking the quotient.
\end{proof}

\hfill

The following trivial lemma describes the $U(1)$-action on total spaces of vector bundles in terms of the $U(1)$-action on functions. We shall omit its proof.

\hfill

\begin{lemma}
\label{lemmab}
Let $E = \mathrm{Spec}_M(\mathrm{S}^\bullet \mathcal F)$ be the total space of a vector bundle over a complex algebraic variety $M$ where $\mathcal F$ is the locally free sheaf on $M$ whose dual is isomorphic to the sheaf of sections of the bundle $E\to M$. Then the $U(1)$-action on $E$ by fiberwise multiplication is induced by the $U(1)$-action on the sheaf of algebras $\mathrm{S}^\bullet \mathcal F$ given by

$$
\lambda\cdot s = \lambda^{-r} s
$$
for $\lambda\in U(1)\subset \C^\times$ and $s$ a local section of $\mathrm{S}^r\mathcal F$.
\end{lemma}

\hfill

We are now ready to give a proof of \ref{very_long_proof}.

\hfill

\pstep Consider the set 
$$
V^\circ := \{x\in V\:|\:x\in V^s\textrm{ and }\operatorname{Stab}_G(x) := \{\id\}\}
$$

By \ref{ss} the geometric quotient of $V^\circ$ exists and is isomorphic to the unique open K\"ahler stratum $N^\circ$ in $N:=V\twice C$. Let $X^\circ:=T^*V^\circ$ be the total space of cotangent bundle to $V^\circ$ and $Z^\circ:=\{x\in X^\circ\:|\: \calM(x) = 0\}$. \ref{_admits_good_quotient_} applied to the free action of $G$ on $V^\circ$ tells us that $Y^\circ := T^*N^\circ$ is isomorphic to $X^\circ\twice G = Z^\circ/G$ as a complex manifold. By \ref{_fine_} all the points of $Z^\circ$ are stable. Hence by \ref{ss} $Z^\circ/G$ is isomorphic to the K\"ahler quotient $Z^\circ\twice C = X^\circ\trice C$ as a complex manifold. One has the following commutative diagram.
$$
\xymatrix{
Y^\circ:=T^*N^\circ \ar[dd]^{\pi_N} && Z^\circ \ar[ll]^{p_Y} \ar@{^{(}->}[r] & X^\circ:=T^*V^\circ \ar[dd]^{\pi_V}\\
& \{x\in X^\circ\:|\:\mu_{hk}(x) = 0\} \ar[ul]^{\rho_Y} \ar@{^{(}->}[ur] &&\\
N^\circ &&& V^\circ \ar[lll]^{p_N}\\
& \{x\in V^\circ\:|\:\mu(x) = 0\} \ar[ul]^{\rho_N} \ar@{^{(}->}[urr]
}
$$
The manifold $Y^\circ$ is a hyperk\"ahler quotient of $X^\circ$. Hence $Y^\circ$ admits a hyperk\"ahler metric (\ref{_red_main_theorem_}). It remains to prove that the constructed metric is indeed the Feix--Kaledin metric.

{\bf Step 2:} Let $\tilde{\Omega}\in\Lambda^{2,0}_I(Y^\circ)$ be the holomorphic symplectic form on $Y^\circ$ induced by the hyperk\"ahler structure. Equivalently, this is the $2$-form on $Y^\circ$ whose pullback to $Z^\circ$ by the quotient map $p_Y$ is $\Omega\restrict{Z^\circ}$. Here $\Omega$ stands as before for the standard holomorphic symplectic form on $X^\circ$. Such a $2$-form exists by \ref{_forms_on_quot_}. 
We need to prove that $\tilde{\Omega}$ is the standard holomorphic symplectic form on the total space of the cotangent bundle. 

Consider the restriction of the tautological $1$-form $\tau$ on $X^\circ$ to $Z^\circ$. In \ref{_hol_symp_ham_} we proved that $\tau\restrict{Z^\circ}$ is $G$-invariant and that $\tau\restrict{Z^\circ}(\xi) = \<\mathcal M\restrict{Z^\circ},\xi\> = 0$ for every $\xi\in\g$. Hence by \ref{_forms_on_quot_} $\tau\restrict{Z^\circ}$ descends to a $1$-form on $Y^\circ = Z^\circ/G$. We shall denote this $1$-form by $\tilde{\tau}$. Now take a point $x\in V^\circ$ and some $\alpha \in Ann(\g_x)\subset T^*_xV$. The covector $\alpha$ can be seen as a point in $Z^\circ$. Denote by $\tilde{x}$ and $\tilde{\alpha}$ the images of $x$ and $\alpha$ under the projection $p_Y\colon Z^\circ \to Y^\circ$. Consider any vector $v\in T_\alpha Z^\circ$. Then
\begin{gather*}
    \tilde{\tau}_{\tilde{\alpha}}(p_{Y*}v) = \tau_\alpha(v) = \alpha(\pi_{V*}v) = \tilde{\alpha}(p_{N*}\pi_{V*}v) = \tilde{\alpha}(\pi_{N*}p_{Y*}v)
\end{gather*}
where the first and the third identity hold as $\tau\restrict{Z^\circ} = p_Y^*\tilde{\tau}$ and $\alpha = p_N^*\tilde{\alpha}$ as a covector. The second identity follows from the definition of $\tau$. We obtain that $\tilde{\tau}$ is the tautological $1$-form on $Y^\circ = T^*N^\circ$. By the construction of the $2$-form $\tilde{\Omega}$ it is equal to $-d\tilde{\tau}$. We see that $\tilde{\Omega}$ is indeed the standard $2$-form on $Y^\circ$.

{\bf Step 3:} Now we shall prove that the $U(1)$-action on $Y^\circ$ by fiberwise multiplication is HKLR-compatible with the constructed hyperk\"ahler metric on $Y^\circ$.

By applying \ref{lemma} to $M = V^\circ$ we introduce a $U(1)$-action on $Y^\circ = X^\circ\trice C$ which is HKLR-compatible with the hyperk\"ahler metric. When we apply \ref{lemmaa} to the case of $S = Z^\circ$, $H = U(1)$ we also obtain a $U(1)$-action on $Y^\circ = Z^\circ/G$. These two actions coincide as the inclusion of $\mu_{hk}^{-1}(0)$ to $Z^\circ$ is $U(1)$-equivariant.

One can see that
$$
\label{abc}
Y^\circ \cong \mathrm{Spec}_{N^\circ}(\mathrm S^\bullet (\mathcal TN^\circ)) \cong \mathrm{Spec}_{N^\circ}(S^\bullet(p_*(\mathcal TV^\circ/\g\cdot\calO_{V^\circ}))^G)
$$
Here the first isomorphism holds as $Y^\circ$ is by definition the total space of the cotangent bundle to $N^\circ$. The second isomorphism follows from \ref{_forms_on_quot_}. With these isomorphisms in mind we use \ref{lemmaa} and \ref{lemmab} to see that the $U(1)$-action on $Y^\circ = T^*N^\circ$ constructed above is the one given by the fiberwise multiplication. We finally proved \ref{very_long_proof}.\endproof{}

\hfill

\remark \ref{very_long_proof} can be stated without the assumption that the action of $C$ on $V$ is generically free as follows.

\hfill

\begin{theorem}
\label{www}
Let $V$ be an Hermitian vector space. Suppose that a complex reductive group $G$ acts on $V$ linearly and the restriction of the action of $G$ to some fixed maximal compact subgroup $C\subset G$ preserves the Hermitian structure. Let $N := V\twice C$ be the K\"ahler quotient of $V$ by $C$ with respect to a fixed moment map $\mu\colon V\to\c^*$. Then the Feix--Kaledin metric is globally defined on the total space $Y_i$ of the cotangent bundle to every K\"ahler stratum $N_i$ of $N$ (see \ref{Sjamaar_Lerman} for the definition of the K\"ahler strata).
\end{theorem}

\hfill

\ref{www} follows easily from \ref{very_long_proof} as the closure of every K\"ahler stratum $N_i\subset N$ is itself a K\"ahler quotient of a vector space (\ref{Sjamaar_Lerman}).

\hfill

\subsection{Metric completions of cotangent bundles}

%%%% LIST OF CHANGES in v.1.2 %%%%%%
% 1. Reformulated everything for the case of the smooth K\"ahler quotient. The previous formulation for singular K\"ahler quotients was a mess.
% 2. Some minor changes
%%%%%%%%

Let $N$ be again a K\"ahler quotient of an Hermitian vector space $V$. Assume that $N$ is smooth i.e. $N=N^\circ$ in the notation of the proof of \ref{very_long_proof}. The manifold $Y:=T^*N$ admits a hyperk\"ahler metric by \ref{very_long_proof}. 

\hfill

\begin{theorem} 
\label{comp}
The metric completion $\tilde{Y}$ of $Y$ has a structure of a stratified hyperk\"ahler space (in the sense of \cite{Mayrand_stratification}).
\end{theorem}

\begin{proof} Consider the hyperk\"ahler quotient $X\trice C$. As $X^\circ:=T^*V^\circ$ is a dense open subset of $X:=T^*V$ we see that $Y = X^\circ\trice C$ embeds into $X\trice C$ as a dense open subset and the embedding preserves stratified hyperk\"ahler structures. It is clear that $X$ is a complete metric space. Hence $X\trice C$ is itself a complete metric space as a quotient of the complete metric space $ \mu^{-1}_{hk}(0)\subset X$ by a compact group $C$. Hence $X\trice C$ is isomorphic to the metric completion of $Y$. \end{proof}

\hfill

\remark In general the space $Y$ is not itself metrically complete. The cases when $Y$ is indeed complete were classified in \cite{Bielawski_Dancer} for $C$ being a torus. By the results of \cite{Bielawski_Dancer} the space $Y = T^*N$ is complete if and only if $N$ is isomorphic to a product of complex projective spaces. The question of completeness will be discussed further in Section~\ref{nef_section}.

\hfill
%%%%%%%%%%

\subsection{Example: Hirzebruch surfaces}
%%%% added this subsection in v.1.2 %%%%% list of changes in v.2.0
% 1. I changed the statement of prop. 3.11 slightly.
% 2. A little bit more of cross-referencing inside the subsection
%%%%%

We shall now illustrate \ref{very_long_proof} and \ref{comp} with a specific example.

\hfill

\begin{definition}
Fix a positive integer $n$. Let the $2$-dimensional torus $ G=\C^\times\times\C^\times$ act on $\C^4$ via
\begin{equation}
\label{hirzebruch}
    (\lambda,\mu)\cdot(x_0, x_1, y_0, y_1) := (\lambda x_0,\lambda x_1,\mu y_0,\mu\lambda^{-n} y_1)
\end{equation}
The quotient of $\C^2\setminus \{0\}\times\C^2\setminus\{0\}\subset \C^4$ by this action is called the {\em Hirzebruch surface} $\Sigma_n$.
\end{definition}

\hfill

Hirzebruch surfaces are ruled surfaces over $\mathbb P^1$. Indeed, consider the map $\Sigma_n\to\P^1$ which sends a point of $\Sigma_n$ represented by a point $(x_0,x_1,y_0,y_1)\in\C^2\setminus \{0\}\times\C^2\setminus\{0\}$ to $[x_0:x_1]\in\mathbb P^1$. One can see that the fibers of this map are isomorphic to $\P^1$.

Our first goal is to represent the Hirzebruch surfaces as K\"ahler quotients. This is done as follows. Consider the action of $C = U(1)\times U(1)$ on $\C^4$ which is the restriction of the action of $\C^\times\times\C^\times$ as in formula~(\ref{hirzebruch}) to $C\subset G$.  By formula~(\ref{_moment_map_unit_rep_}) the moment map $\mu\colon \C^4 \to \c^* = \R^2$ is given by
\begin{equation}
\label{mu}
    \mu(x, y) = -\frac{1}{2}(||x||^2-n|y_1|^2 - c_0, ||y||^2 - c_1)
\end{equation}
where $x = (x_0,x_1)\in \C^2$, $y = (y_0,y_1)\in\C^2$ and $c_0,c_1$ are some real numbers. From now on we assume that $c_0$ and $c_1$ are both positive.

Let us now describe the sets of stable and semistable points of $\C^4$. Recall that by \ref{_description_of_unstable_} they are equal to

\begin{equation}
\label{reminder}
V^{s} =  V\setminus\left(\bigcup\limits_{\<\theta,\xi\> \le 0} V^{\xi\ge 0}\right);\:\:\:\:\:
V^{ss} = V\setminus\left(\bigcup\limits_{\<\theta,\xi\> < 0} V^{\xi\ge 0}\right)
\end{equation}

where the unions are taken over non-zero $\xi = (a,b)\in \c = \R^2$. In our case $\theta\in\c^* = \R^2$ is equal to $\frac{1}{2}(c_0,c_1)$. The inequality $\<\theta,\xi\> = \frac{1}{2}(c_0a+c_1 b)\le 0$ implies that either $a$ or $b$ is negative. The element $\xi = (a,b)$ of the Lie algebra $\c$ acts on $\C^4$ via the linear operator $\sqrt{-1}\diag(a,a,b,b-na)$ by formula~(\ref{hirzebruch}). Table~\ref{zalupa} describes the subspaces $(\C^4)^{\xi\ge 0}$ for all possible $\xi = (a,b)$ such that $\<\theta,\xi\>\le 0$, according to inequalities on $a$ and $b$.

\begin{table}
\caption{Unstable subspaces of $\C^4$}
\label{zalupa}
\begin{tabular}{|c|c|c|c|c|}
    \hline
    Inequalities on $a$ and $b$ & \multicolumn{3}{|c|}{Signs of the eigenvalues of $-\sqrt{-1}\xi$} & $(\C^4)^{\xi\ge 0}$ \\
    &$a$& $b$ & $b-na$ &\\
    \hline
    $a\ge 0 > b$ & $\ge 0$ & $< 0$ & $< 0$ &$\{(x_0,x_1,0,0\}$\\
    \hline
    $b\ge 0 > a$ & $<0$ & $\ge 0$ & $>0$ & $\{(0,0,y_0,y_1)\}$\\
    \hline
    $0>b\ge na$ & $<0$ & $<0$ & $\ge 0$ & $\{(0,0,0,y_1)\}$\\
    \hline
    $0>na>b$ & $<0$ & $<0$ & $<0$ & $\{(0,0,0,0)\}$\\
    \hline
\end{tabular}
\end{table}

One can see from Table~\ref{zalupa} and formula~(\ref{reminder}) that the sets of stable and semistable points coincide and are as follows:
\begin{equation}
\label{kakzheetovse}
    (\C^4)^s = (\C^4)^{ss} = \C^4\setminus (\{(x_0,x_1,0,0)\}\cup \{(0,0,y_0,y_1)\} = \C^2\setminus\{0\}\times\C^2\setminus\{0\}
\end{equation}
\ref{ss} implies that the K\"ahler quotient $\C^4\twice C$ is the Hirzebruch surface $\Sigma_n$.

We notice that the stabilizer of every point of $(\C^4)^s = \C^2\setminus\{0\}\times\C^2\setminus\{0\}$ is trivial. Hence the set $(\C^4)^\circ$ of stable points with trivial stabilizer considered in the proof of \ref{very_long_proof} coincides with $(\C^4)^s$.

If one wants to construct the variety $T^*\Sigma_n$ one should take the hyperk\"ahler quotient $T^*(\C^4)^\circ\trice C$. Our second goal in this subsection is to show that the variety $T^*\Sigma_n = T^*(\C^4)^\circ\twice G $ is strictly contained in the hyperk\"ahler quotient $T^*(\C^4)\trice C$ and to describe the complement of $T^*\Sigma_n$ in $T^*(\C^4)\trice C$.

The action of $G = \C^\times\times\C^\times$ on $T^*\C^4$ is given by
\begin{multline}
\label{joppa}
    (\lambda,\mu)\cdot (x_0,x_1,y_0,y_1,z_0,z_1,w_0,w_1) = \\
    = (\lambda x_0,\lambda x_1,\mu y_0,\mu\lambda^{-n}y_1,\lambda^{-1}z_0,\lambda^{-1}z_1,\mu^{-1}w_0,\mu^{-1}\lambda^nw_1)
\end{multline}

\hfill

We want to describe the sets of stable and semistable points of $T^*\C^4$. This problem is solved analoguously to the problem of description of stable and semistable points of $\C^4$. The element $\xi = (a,b)$ of the Lie algebra $\c$ acts on $T^*\C^4$ via the linear operator $\sqrt{-1}\diag(a,a,b,b-na,-a,-a,-b,na-b)$ by formula~(\ref{joppa}). Table~\ref{krugomhui} describes the subspaces $(T^*\C^4)^{\xi\ge 0}$ for all possible $\xi = (a,b)$ such that $\<\theta,\xi\>\le 0$, according to inequalities on $a$ and $b$.
\begin{table}
\caption{Unstable subspaces of $T^*\C^4$}
\label{krugomhui}
\begin{tabular}{|p{1.67cm}|p{0.55cm}|p{0.55cm}|p{0.55cm}|p{0.55cm}|p{0.55cm}|p{0.55cm}|c|}
    \hline
     Inequalities on $a$ and $b$ &  \multicolumn{6}{|c|}{ Signs of the eigenvalues of $-\sqrt{-1}\xi$} & $(T^*\C^4)^{\xi\ge 0}$\\ 
     & $a$ & $b$ & $b-na$ & $-a$ & $-b$ & $na-b$ &\\
     \hline
     $a>0>b$ & $>0$ & $<0$ & $<0$ & $<0$ & $>0$ & $>0$ & $\{(x_0,x_1,0,0,0,0,w_0,w_1)\}$ \\
     \hline
     $a = 0 >b$ & $=0$ & $<0$ & $<0$ & $= 0$ & $>0$ & $>0$ & $\{(x_0,x_1,0,0,z_0,z_1,w_0,w_1)\}$\\
     \hline
     $b>0>a$ & $<0$ & $>0$ & $>0$ & $>0$ & $<0$ & $<0$ & $\{(0,0,y_0,y_1,z_0,z_1,0,0)\}$\\
     \hline
     $b = 0>a$ & $<0$ & $=0$ & $>0$ & $>0$ & $=0$ & $<0$ & $\{(0,0,y_0,y_1,z_0,z_1,w_0,0)\}$\\
     \hline
     $0>na>b$ & $<0$ & $<0$ & $<0$ & $>0$ & $>0$ & $>0$ & $\{(0,0,0,0,z_0,z_1,w_0,w_1)\}$\\
     \hline
     $0>na=b$ & $<0$ & $<0$ & $= 0$ & $>0$ & $>0$ & $= 0$ & $\{(0,0,0,y_1,z_0,z_1,w_0,w_1)\}$\\
     \hline 
     $0>b>na$ & $<0$ & $<0$ & $>0$ & $>0$ & $>0$ & $<0$ & $\{(0,0,0,y_1,z_0,z_1,w_0,0)\}$\\
     \hline
\end{tabular}
\end{table}

The set of unstable points is the union of vector subspaces $(T^*\C^4)^{\xi\ge 0}$ in the last column of Table~\ref{krugomhui}. One can see from formula~(\ref{reminder}) that the sets of stable and semistable points of $T^*\C^4$ coincide and
\begin{multline}
\label{nadoekhat}
    T^*\C^4\setminus(T^*\C^4)^s = (T^*\C^4)^{us} =\\
    = \{(x_0,x_1,0,0,z_0,z_1,w_0,w_1)\}\cup \{(0,0,y_0,y_1,z_0,z_1,w_0,0)\}\cup \\
    \cup\{(0,0,0,y_1,z_0,z_1,w_0,w_1)\}
\end{multline}
Consider the set $E$ of points of $T^*\C^4$ which are {\em (semi)stable} but whose images in $\C^4$ under the natural projection $p$ to $\C^4$ are {\em not (semi)stable}. Of course, it does not matter if one considers stability or semistability because in our case these notions coincide. By formula~(\ref{kakzheetovse}) and formula~(\ref{nadoekhat}) the set $E$ is {\em non-empty} and is described as follows:
\begin{multline}
\label{navernopotomuchto}
    E = \left(p^{-1}(\C^4)^{us}\right)\setminus (T^*\C^4)^{us} 
    = \\
    = \{(0,0,y_0,y_1,z_0,z_1,w_0,w_1)\:|\: y_0\ne 0 \textrm{ and } w_1\ne 0\}
\end{multline}

Consider the set $Z^s$ of stable points of $T^*\C^4$ on which the holomorphic symplectic map $\calM\colon T^*\C^4\to \g^* = \C^2$ vanishes. The set $Z^s$ contains a dense open subset $Z^\circ = \{x\in Z^s\:|\: p(x)\in (\C^4)^s = (\C^4)^\circ\}$. The quotient of $Z^\circ$ by $G$ is nothing but $T^*\Sigma_n = T^*(\C^4)^\circ\trice C$. The complement of $T^*\Sigma_n$ in the hyperk\"ahler quotient $T^*\C^4\trice C$ is thus isomorphic to the quotient of $Z^s\cap E$ by $G$ as a complex variety.

Let us first describe the variety $Z^s\cap E$. By \ref{_hol_symp_ham_} the moment map $\calM\colon T^*\C^4\to \C^2$ is as follows:
\begin{equation}
\label{vseetomoichuvstva}
    \calM(x_0,x_1,y_0,y_1,z_0,z_1,w_0,w_1) = \sqrt{-1}(x_0z_0 + x_1z_1 - ny_1w_1, y_0w_0 + y_1w_1)
\end{equation}

Consider the restriction of $\calM$ to $E$. It is given as
\begin{equation}
    \calM(0,0,y_0,y_1,z_0,z_1,w_0,w_1) = \sqrt{-1}(-ny_1w_1, y_0w_0 + y_1w_1)
\end{equation}
A point of $E$ thus lies in $Z^s$ if and only if $y_0w_0 = y_1w_1 = 0$. By formula~(\ref{navernopotomuchto})
\begin{equation}
    E\cap Z^s = \{(0,0,y_0,0,z_0,z_1,0,w_1)\:|\: y_0\ne 0 \textrm{ and } w_1\ne 0\}
\end{equation}

To describe the quotient of $E\cap Z^s$ by the action of $G$ we notice that every $G$-orbit of $E\cap Z^s$ intersects with the set $S := \{y_0 = w_1 = 1\}$. Two points $p = (0,0,1,0,z_0,z_1,0,1)$ and $p' = (0,0,1,0,z_0',z_1',0,1)$ of $S$ lie in the same $G$-orbit if and only if $(z_0',z_1') = \lambda (z_0,z_1)$ for some $n$-th root of unity $\lambda$. Therefore 
\begin{equation}
    (E\cap Z^s)/G = S/\mu_n \cong \C^2/\mu_n
\end{equation}

We have proved the following.

\hfill

\begin{proposition}
The manifold $T^*\Sigma_n$ admits the globally defined Feix--Kaledin metric. This metric is non-complete. The metric completion of $T^*\Sigma_n$ is the hyperk\"ahler quotient $(T^*\C^4)\trice C$. The complement of $T^*\Sigma_n$ in $(T^*\C^4)\trice C$ is isomorphic to $\C^2/\mu_n$ as a complex variety. In particular, the underlying complex variety of $(T^*\C^4)\trice C$ is non-smooth and has quotient singularities.
\end{proposition}

\hfill
%%%%%%%%%%%%%%%%%%%%%%%%%%%%%%%
%%%%%%%%%%%%%%%%%%%%%%%%%%%%%%%
\section{Cotangent Bundles as Stein and Affine Varieties}

\subsection{Twisted complex structure on cotangent bundle}

Consider the vector space $X_\R = V_\R\oplus V_\R$ as a real vector space with the complex structure $J$. Recall that
\begin{equation}
\label{j}
J(x,y) = (-y,x)
\end{equation}

We draw the reader's attention to the fact that the (semi)stability of a point $x\in X$ in general does depend on the choice of a complex structure on $X$. In this subsection we will consider points of $X$ which are semistable with respect to the complex structure $J$, the K\"ahler form $\omega_J$ and the moment map $\mu_J = \Re\mathcal M$. We shall call these points {\em J-semistable}.

\hfill

\begin{proposition}
\label{j_semistable}
Every point of $X$ is $J$-semistable with respect to the moment map $\mu_J = \Re \mathcal M$.
\end{proposition}

\begin{proof} We shall prove the statement using the Hilbert-Mumford criterion i.e. by computing the $\mu$-weights of points of $X$ with respect to the complex structure $J$ (\ref{mu_weight_def}). For a point $(x,y)\in X = V_\R\oplus V_\R$ and $\xi\in\c$ we shall denote the corresponding $\mu$-weight by $w_J^\xi(x,y)$. Let $\xi\in \c$ be a skew-symmetric operator acting on $V$. Then the matrix of its action on $X$ is
$$
\xi_X = \begin{pmatrix}
\xi & 0\\
0 & \xi
\end{pmatrix}
$$
and
$$
J\xi_X = \begin{pmatrix}
0 & -\xi\\
\xi & 0
\end{pmatrix}
$$
The exponent of $J\xi_X$ is given by
\begin{equation}
\label{exp}
    \exp(tJ\xi_X) = 
    \begin{pmatrix}\cos t\xi & -\sin t\xi\\ \sin t\xi & \cos t\xi\end{pmatrix}
\end{equation}
Now let $\sqrt{-1}\lambda_i$ be (imaginary) eigenvalues of the operator $\xi$ as in \ref{Hilbert_Mumford_ex}. We can decompose every vector $x\in V$ into the sum $x = \sum x_i$ of eigenvectors of $\xi$. If we consider $V$ as a real vector space with the complex structure $I$ then
$$
\xi x_i = \lambda_i Ix_i
$$
hence 
\begin{equation}
\label{zzz}
\cos(t\xi) x_i= \cosh(t\lambda_i) x_i, \:\:\:\:\: \sin(t\xi) x_i = \sinh(t\lambda_i) Ix_i
\end{equation}
Let
\begin{equation}
    (x(t),y(t)):= \exp(tJ\xi_X)(x,y)
\end{equation}
Then due to formulae~(\ref{exp}) and (\ref{zzz})
\begin{gather}
    x(t) = \sum\limits_i \cosh(t\lambda_i) x_i - I\sum\limits_i \sinh(t\lambda_i) y_i\\ 
    y(t) = I \sum\limits_i \sinh(t\lambda_i) x_i + \sum\limits_i \cosh(t\lambda_i) y_i
\end{gather}

By \ref{mu_weight_def} and formula~(\ref{_moment_map_cot_}) the $\mu$-weights are equal to
\begin{equation}
\label{weights_a}
w_J^\xi(x,y) = \lim\limits_{t\to\infty}\<\xi x(t),y(t)\>
\end{equation}
We now use formula~(\ref{weights_a}) to compute them explicitly. The vector $\xi x(t)$ is equal to
\begin{gather*}
    \xi x(t) = I\sum\limits_i\lambda_i\cosh(t\lambda_i) x_i + \sum\limits_i\lambda_i\sinh(t\lambda_i)y_i
\end{gather*}
Hence the scalar product of $\xi x(t)$ with the vector $y(t)$ equals
\begin{gather*}
    \sum\limits_i \lambda_i \left[\frac{1}{2}\sinh (2t\lambda_i)(||x_i||^2 + ||y_i||^2) + \cosh(2t\lambda_i)\omega(x_i,y_i)\right]
\end{gather*}

Computing the limit of the sum is an exercise in a real calculus. As long as $\lambda_i\ne 0$ and at least one of the vectors $x_i$ and $y_i$ is non-zero, the $i$-th summand tends either to $+\infty$ or to $0$. The last case occurs if and only if $x_i =\mathrm{sgn}(\lambda_i) Iy_i$. In particular, $w_J^\xi(x,y)$ is always non-negative. By Hilbert-Mumford criterion (\ref{Hilbert_Mumford}) every point of $X$ is $J$-semistable. 
\end{proof}

\hfill

As a corollary of the proposition above we obtain the following

\hfill

\begin{theorem}
\label{affine}
Let $\tilde{Y} = X\trice C$ be a hyperk\"ahler quotient of $X = T^*V$. Let $J\in\mathbb H$ be a complex structure on $\tilde{Y}$ different from $I$ and $-I$. Then the complex analytic variety $\tilde{Y}_J$ is an affine algebraic variety. Moreover, for every two complex structures $J_1$ and $J_2$ different from $\pm I$ the varieties $\tilde{Y}_{J_1}$ and $\tilde{Y}_{J_2}$ are isomorphic as algebraic varieties. 
\end{theorem}

\pstep Let $J$ be the complex structure on the vector space $X$ as defined in formula~(\ref{j}). The action of $C$ on $X$ preserves $J$ and can be uniquely extended to the $J$-holomorphic action of the complexified group $G$ on $X_J$ as in formula~(\ref{exp}). Consider the holomorphic symplectic form $\Omega_J = \omega_K + \sqrt{-1}\omega_I \in \Lambda^{2,0}_J X$. It is preserved by the action of $C$ and hence by the $J$-holomorphic action of $G$. Indeed, for every $\xi\in\c$
\begin{equation}
    \mathsf L_{J\xi} \Omega_J = d(J\xi\cntrct\Omega_J) = \sqrt{-1}d(\xi\cntrct\Omega_J) = \sqrt{-1}\mathsf L_\xi\Omega_J = 0
\end{equation}
as $\Omega_J$ is closed.

Consider the $J$-holomorphic moment map $\mathcal M_J\colon X \to \g^*$. Explicitly, for every $\xi\in\c$
\begin{equation}
    \<\mathcal M_J,\xi\> = \<\Im \mathcal M, \xi\> + \sqrt{-1}\<\mu_I,\xi\>
\end{equation}
where $\mu_I\colon X\to\c^*$ is the moment for the action of $C$ with respect to the complex structure $I$.

By \ref{j_semistable} every point of the closed algebraic $G$-invariant subvariety $\mathcal M_J^{-1}(0)$ is $J$-semistable. Hence by Proposition~\ref{ss} 
\begin{equation}
    \tilde{Y}_J = \mathcal M_J^{-1}(0)\twice_{\mu_J} C = \mathcal M_J^{-1}(0)/G
\end{equation}
As $\mathcal M_J^{-1}(0)$ is affine its good quotient $\tilde{Y}_J$ is again affine by \ref{_existence_of_quotient_affine_}.

{\bf Step 2:} Consider the action of $\C^\times$ on $X$ by the fiberwise multiplication. For every two quaternionic complex structures $J_1,J_2$ not equal to $\pm I$ there exists $\lambda\in\C^\times$ such that $\lambda_* J_1 = J_2$ (\cite{HKLR}). Hence $\lambda$ is a complex linear $C$-equivariant isomorphism between $(X,J_1)$ and $(X,J_2)$. Moreover, $\lambda$ is an isomorphism of holomorphic symplectic vector spaces. If $\calM_i$ denotes the holomorphic symplectic moment map on $X_{J_i}$ then one can check that $\lambda^*\calM_2 = \calM_1$ (\cite{HKLR}). As a consequence $\lambda$ induces an algebraic $G$-eqivariant isomorphism of $\tilde{Y}_{J_i} = \calM_i^{-1}(0)/G$,   $i=1,2$.
\endproof{}

\hfill

%%%%%%%%%%%%%%%
%%%%%%%%%%%%%%%

\subsection{Stein structure}

Throughout this subsection we assume that the hyperk\"ahler quotient $\tilde{Y} = X\trice C$ is smooth. We present here another approach to \ref{affine}. This approach does not let us to prove that $\tilde{Y}_J$ is affine but only that it is Stein. However, it is interesting in its own way and possibly can be applied in a more general situation.

\hfill

\begin{definition} (\cite{Demailly}, {\sf Ch. I \S6, Ch. VII \S9})
\label{Stein_def}
Let $M$ be a complex manifold. It is called {\em Stein} if one of the following equivalent conditions holds:
\begin{itemize}
    \item[(i)] $M$ admits a closed holomorphic embedding to an affine complex space.
    \item[(ii)] There exists a smooth proper strictly plurisubharmonic function $\rho\colon M \to\R_{\ge 0}$ (see the definition below).
\end{itemize}
\end{definition}

\hfill

\begin{definition}
\label{plurisubharmonic}
An $\R$-valued $C^2$-function on a complex manifold $M$ is called {\em strictly plurisubharmonic} if $dd^c\rho$ is a strictly positive $(1,1)$-form.
\end{definition}

\hfill

\remark A (possibly singular) {\em Stein space} can be defined as a complex space admitting a closed embedding into $\C^N$. An equivalent definition in terms of strictly plurisubharmonic functions also does exist (\cite{_Fornaess_Narasimhan_}). However, defining strictly plurisubharmonic functions in terms of the $\overline{\partial}$-operator seems to be a subtle question. Its discussion will lead us too far beyond our topic. That's why we prefer to impose a rather restrictive smoothness assumption in our exposition.

\hfill

\remark Affine varieties are Stein. The converse is of course not true. Even more, there exist examples of Stein algebraic varieties which are not affine (\cite{Hartshorne}, {\sf Ch. VI, Ex. 3.2}).

\hfill

Following \cite{HKLR} we are going to construct a proper strictly plurisubharmonic function on $\tilde{Y}_J$.

\hfill

\begin{lemma}
\label{stages}
The $U(1)$-action on $\tilde{Y}_I$ is Hamiltonian.
\end{lemma}

\begin{proof} Being linear, the $U(1)$-action on $X$  is Hamiltonian. In addition, it commutes with $C$. Hence by the technique of reduction in stages (\cite{Cannas_da_Silva}, Part IX \S 24.3) the $U(1)$-action stays Hamiltonian after taking the quotient. 
\end{proof}

\hfill

\begin{proposition}
\label{stein}
Suppose that the K\"ahler quotient $N = V\twice C$ is smooth and compact. Let $\psi$ be a moment map for the $U(1)$ action on $\tilde{Y}$. Then $-\psi$ is a proper strictly plurisubharmonic function on $\tilde{Y}_J$.
\end{proposition}

\pstep Let $X$ be an arbitrary hyperk\"ahler manifold equipped with a Hamiltonian $U(1)$-action rotating the complex structures. Denote by $\phi$ the vector field tangent to the $U(1)$ action. Then
\begin{gather}
    \mathsf L_\phi\omega_I = 0\\
    \mathsf L_\phi\omega_J = \omega_K\\
    \mathsf L_\phi\omega_K = -\omega_J
\end{gather}
The moment map $\psi\colon X\to\R$ satisfies
\begin{equation}
    d\psi = \phi\cntrct\omega_I
\end{equation}
As in \cite{HKLR} we compute
\begin{gather*}
    dd^c_J\psi = dJd\psi = d J (\phi\cntrct\omega_I) = d((J\phi)\cntrct\omega_I) = d(\phi\cntrct \omega_K) = \mathsf L_\phi\omega_K = -\omega_J
\end{gather*}
Hence $-\psi$ is a K\"ahler potential for $X_J$, in particular, this function is strictly plurisubharmonic.

{\bf Step 2:} The moment map $\Psi\colon X\to \R$ for the $U(1)$-action on $X = V_\R\oplus V_\R$ is given by
$$
\Psi(x,y) = -\frac{1}{2}||y||^2
$$
It is a non-positive function. By \ref{stages} it descends to a moment map $\psi$ on $\tilde{Y}$. Hence the function $\rho:=-\psi$ is non-negative and strictly plurisubharmonic. Now we need to check that $\rho$ is proper. 

Fix a number $r\in\R$, $r>0$ and consider the closed set $\rho^{-1}([0,r])\subset \tilde{Y}$. It is the quotient by a compact group of the following closed subset of $X$ 
$$
K_r:=\{(x,y)\in X|\:\mu_{hk}(x,y) = 0, -\Psi(x,y)\le r\}
$$
In particular, for $(x,y)\in K_r\subset V_\R\oplus V_\R$
\begin{gather}
    \mu_I(x,y) = 0
\end{gather}
and $||y||^2\le 2r$. By \ref{_direct_sum_} for every $(x,y)\in K_r$ and $\xi\in \c$
\begin{equation*}
    |\<\mu(x),\xi\>| = \left|\frac{\omega(x,\xi x)}{2} + \<\theta,\xi\>\right| = \left|\frac{\omega(y,\xi y)}{2}\right|\le \frac{1}{2}||\xi||_V||y||^2\le r||\xi||_V
\end{equation*}
where $||\xi||_V$ is the norm of $\xi$ as an operator on $V$. Hence the claim will follow if we prove that the moment map $\mu\colon V\to\c^*$ is proper.

{\bf Step 3:} The properness of $\mu$ follows straightforwardly from \cite{Sjamaar_convexity}, Lemma 4.10. For the sake of completeness we include the proof here.

Suppose $\{u_n\}$ is an unbounded sequence of vectors of $V$. We want to check that the sequence $\{\mu(u_n)\}$ is also unbounded. By passing to a subsequence we may assume that $\lim\limits_{n\to\infty}||u_n||$ exists and is equal to $+\infty$. The subset of points of the form $gv$ where $g\in G$ and $\mu(v) = 0$ is dense in $V$, hence without loss of generality we may and shall assume that the points $u_n = g_n v_n$ are of this form.

Fix a maximal torus $T\subset C$ with the Lie algebra $\mathfrak t$. The reductive group $G$ admits the polar decomposition $G = C\exp(\sqrt{-1}\mathfrak t) C$. One may write $g_n = k_n\exp(\sqrt{-1}\xi_n)h_n$ where $k_n, h_n\in C$ and $\xi_n\in \mathfrak t$. As the subset $\mu^{-1}(0)$ is $C$-invariant one may assume that $h_n=1$ by changing $v_n$ if necessary. As the action of $C$ on $V$ changes neither the norm of a vector nor the norm of its image under $\mu$ one may assume that $k_n$ is also the unit element. Hence we are reduced to the case when $u_n = \exp(\sqrt{-1}\xi_n)v_n$ where $\xi_n\in \mathfrak t$ and $\mu(v_n)=0$.

Take any orthonormal basis $\{e^i\}$ of $V$ in which the action of $T$ is diagonalized. Then there exist $\beta^i\in \mathfrak t^*$ such that for every $\xi\in\mathfrak t$
$$
\xi\cdot e^i = \sqrt{-1}\beta^i(\xi)e^i
$$
Decompose $v_n = \sum\limits_i v_n^i e^i$ with respect to this basis. Then
$$
||g_n v_n||^2 = \sum\limits_i\exp(-2\beta^i(\xi_n))|v_n^i|^2
$$
As by our assumptions the sequence $||g_nv_n||^2$ is unbounded, the subset of indices $A:= \{i\: |\: \text{the sequence } \exp(-2\beta^i(\xi_n))|v_n^i|^2 \text{ is unbounded }\}$ is non-empty. By passing to a subsequence if necessary one may assume that $\forall i\in A$ the limit $\lim\limits_{n\to\infty} \exp(-2\beta^i(\xi_n))|v_n^i|^2$ exists and is equal to $+\infty$. The set $A$ might become smaller but it will remain non-empty.

We assumed that the quotient $V\twice C$ is compact. That is equivalent to the compactness of $\mu^{-1}(0)$. Hence $|v_n^i|^2$ is bounded $\forall i$. Consequently, $\beta^i(\xi_n)$ tends to $-\infty$ for every $i\in A$. In particular $\exists \eta\in\mathfrak t$ of unit norm such that $\forall i\in A$ 
$$
\beta^i(\eta) < 0
$$
Now one can estimate
$$
    ||\mu(g_nv_n)||\ge |\<\mu(g_n v_n),\eta\>| = \left|\frac{1}{2}\sum\limits_i\beta^i(\eta)\exp(-\beta^i(\xi_n))|v_n^i|^2 + \<\theta,\eta\>\right|
$$
for some $\theta\in\c^*$ depending only on the choice of $\mu$. By the choice of $\eta$ every summand in $\sum\limits_i\beta^i(\eta)\exp(-\beta^i(\xi_n))|v_n^i|^2$ is either bounded or tends to minus infinity. Hence $\mu(g_nv_n)$ is unbounded.\endproof{}

\hfill

\subsection{Big and nef tangent bundles}
\label{nef_section}
%%%%%% List of changes in v.1.1
% 1. Removed lemma 4.8. and proposition 4.9. outside of the proof of theorem 4.10. Stated the result of Pedersen-Poon explicitly in prop 4.9
% 2. Stated and proved in Corollary 4.11 that the tangent bundle is also big
% 3. Added a discussion of CP conjecture
%%%%% list of changes in v.1.2
% 1. Changed "semistable" to "stable" everywhere

Let $N$ be a K\"ahler manifold and $E$ a vector bundle over $N$. The space of all hyperplanes in $E$ will be called the {\em projectivization} of $E$ and denoted by $\mathbb P(E)$. Clearly the manifold $\mathbb P(E)$ is isomorphic to $\Tot(E^*)/\C^\times$. There is the line bundle $\mathcal O(1)$ over $\mathbb P(E)$ defined similarly to the line bundle $\mathcal O(1)$ over a projective space.

\hfill

\begin{definition}
\label{nef_def}
\begin{itemize}
    \item[(i)] A line bundle $L$ over $N$ is called a {\em nef line bundle} if its first Chern class $c_1(L)\in H^{1,1}(X,\mathbb Z)$ lies in the closure of the K\"ahler cone of $N$.
    \item[(ii)] A line bundle $L$ over $N$ is called a {\em big line bundle} if 
    \begin{equation}
        h^0(N, L^{\otimes d}) = O(d^n)
    \end{equation}
    where $n = \dim N$.
    \item[(iii)] A vector bundle $E$ over $N$ is called a {\em nef (resp. big) vector bundle} if the line bundle $\mathcal O(1)$ over $\mathbb P(E)$ is a nef (resp. big) line bundle.
\end{itemize}
\end{definition}

\hfill

To determine whether a given vector bundle $E$ is nef (or big) we need some information about the manifold $\mathbb P(E)$. The next lemma is concerned with the case $E = TN$. Let us fix notation before stating the lemma. For any manifold $Z$ of dimension $m$ (resp. a smooth map $f\colon Z\to Z'$ with fibers of equal dimension $m$) let $K_Z := \Lambda^m \Omega_Z$ (resp. $K_{Z/Z'} := \Lambda^m\Omega_{Z/Z'}$) denote the canonical line bundle of $Z$ (resp. the relative canonical line bundle of the map $f\colon Z\to Z'$).

\hfill

\begin{lemma}
Let $N$ be a complex manifold of dimension $n$. Then the canonical line bundle $K_{\mathbb P(TN)}$ on $\mathbb P(TN)$ is isomorphic to the line bundle $\mathcal O(-n)$.
\end{lemma}

\begin{proof}
Let $Z$ denote the manifold $\mathbb P(TN)$ and $\pi\colon Z\to N$ the natural projection. First, one has the relative Euler exact sequence of sheaves on $Z$
\begin{equation}
    0 \to \Omega_{Z/N} \to \pi^*\mathcal TN\otimes \mathcal O(-1) \to \mathcal O_Z\to 0
\end{equation}
By taking the highest exterior degree of the terms of the sequence we get
\begin{equation}
\label{179}
    \pi^*K_N^*\otimes \mathcal O(-n) \cong K_{Z/N}
\end{equation}
Second, we have the short exact sequence of differentials
\begin{equation}
    0 \to \pi^*\Omega_N \to \Omega_Z \to \Omega_{Z/N}\to 0
\end{equation}
hence 
\begin{equation}
\label{57}
K_Z = \pi^* K_N\otimes K_{Z/N}
\end{equation}
By putting formulae~(\ref{179}) and (\ref{57}) together we obtain that
\begin{equation}
    K_Z \cong \pi^*K_N\otimes \pi^*K_N^*\otimes \mathcal O(-n) \cong \mathcal O(-n)
\end{equation}
\end{proof}

Before going any further we shall state a part of \cite{Pedersen_Poon}, {\sf Prop. 1.10} in a form convenient for our exposition. This result will turn out to be useful to study the geometry of the manifold $Z = \mathbb P(TN)$.

\hfill

\begin{proposition} 
\label{_pedersen_poon_}
(\cite{Pedersen_Poon}) Let $Y$ be a K\"ahler Ricci-flat manifold equipped with a Hamiltonian action of $U(1)$. Fix a moment map $\psi_{t_0}\colon Y\to\mathbb R$ and for any $t\in\R$ let $\psi_t := \psi_{t_0} + (t-t_0)$ be another moment map. Assume that $0$ is a regular value of $\psi_{t_0}$.
For any point $y\in\psi^{-1}(0)$ consider a neighbourhood $Y'\subset Y$ of $y$. Let us call points of $Y'$ equivalent if they lie in the same $\C^\times$-orbit. Assume that $Y'$ satisfies the following property: for any $t\in\mathbb R$ close enough to zero the set $\psi_t^{-1}(0)$ is contained in $Y'$ and the natural map
$$
\psi_t^{-1}(0)/U(1) \to Y'/\sim
$$
is an isomorphism. Let $\omega(t)$ be the K\"ahler form on $Z: = Y'/\sim$ obtained from K\"ahler reduction with respect to $\psi_t$. Then the family $\omega(t)$ of forms satisfies the following differential equation

\begin{equation}
    c\frac{d}{dt}\omega(t) = \rho(t) - \sqrt{-1}\partial\overline{\partial}\log\delta(t)
\end{equation}
where $\rho(t)$ is the Ricci form of $\omega(t)$, $\delta(z, t)$ is the length of the circle on the submanifold $\{\mu(x) + t = 0\} \subset Y$ lying over a point $z\in Z$, and $c$ is a constant.
\end{proposition}\endproof{}

\hfill

We shall now state a general theorem which relates the properties of the Feix--Kaledin metric on $Y = T^*N$ with the geometry of $N$. This theorem is due to R. Bielawski who kindly permitted me to write his proof here.

\hfill

\begin{theorem}
\label{nef_Bielawski}
({\sf R. Bielawski, private correspondence}) Let $N$ be a K\"ahler manifold of dimension $n$. Suppose that the Feix--Kaledin metric is defined globally on $Y = T^*N$. Assume that the moment map $\psi\colon Y\to\R$ for the $U(1)$-action on $Y$ is proper. Then the tangent bundle $TN$ to the manifold $N$ is nef.
\end{theorem}

\begin{proof}{\bf Step 1:} Choose the moment map $\psi\colon Y\to\R_{\le 0}$ for the $U(1)$-action on $Y$ in such a way that $N = \psi^{-1}(0)$. For any $t>0$ let $\psi_t:=\psi + t$ be another moment map for the $U(1)$-action on $Y$. Denote by $Y\twice^t U(1)$ the K\"ahler quotient of $Y$ with respect to the moment map $\psi_t$. We start by showing that $Y\twice^t U(1)$ is naturally isomorphic to $\mathbb P(TN)$ as a complex manifold. We do this by describing the set $Y^{s}_t$ of stable points of $Y$ with respect to the moment map $\psi_t$ with the use of the Hilbert-Mumford criterion (\ref{Hilbert_Mumford}). We claim that for any $t>0$
\begin{equation}
    Y^{s}_t = Y\setminus N
\end{equation}
where $N$ is embedded to $Y$ as the zero section. Let $\xi\in\mathfrak u(1)\cong \R$ be the unit positive pointing vector. For any $x\in Y$ we shall denote by $w^+_t(x)$ (resp. $w^-_t(x)$) the weight of the pair $x$ and $\xi$ (resp. $-\xi$) with respect to the moment map $\psi_t$ (\ref{mu_weight_def}). We compute the weights:
\begin{equation}
    w^+_t(x) = \lim\limits_{s\to\infty} \psi_t(e^{-s}x) = \lim\limits_{s\to\infty} \psi(e^{-s}x) +t = t>0
\end{equation}
as $\lim\limits_{s\to\infty} e^{-s}x$ is a point of the zero section and $\psi$ vanishes on the zero section.
\begin{equation}
    w^-_t(x) = -\lim\limits_{s\to\infty} \psi_t(e^s x) = -\lim\limits_{s\to\infty} \psi(e^s x) -t = \begin{cases}
    +\infty,&\text{if } x\notin N\\
    -t<0,&\text{otherwise}
    \end{cases}
\end{equation}
Indeed, $\psi\le 0$ and it cannot be bounded on the ray $\{e^s x\}$ as this map is proper. We obtain that $Y^{s}_t = Y\setminus N$ (\ref{Hilbert_Mumford}). By \ref{ss} the good quotient $(Y\setminus N)/\C^\times$ is isomorphic to the K\"ahler quotient $Y\twice^t U(1)$ as a complex variety. The quotient $(Y\setminus N)/\C^\times$ is equal to $\mathbb P(TN)$ by the very definition of the projectivization of a vector bundle, hence the claim.

{\bf Step 2:} In Step 1 we identified each K\"ahler quotient $Y\twice^t U(1)$ with $Z$. As $Y\twice^t U(1)$ is equipped with a K\"ahler form $\omega(t)$, we obtain a one-parameter family $\omega(t), t>0$ of K\"ahler forms on $Z$. The manifold $Y$ is Ricci-flat as it is hyperk\"ahler. We may now apply \ref{_pedersen_poon_} with $Y' = Y^{s} = Y\setminus N$ to obtain

\begin{equation}
    c\frac{d}{dt}\omega(t) = \rho(t) - \sqrt{-1}\partial\overline{\partial}\log\delta(t)
\end{equation}

In \cite{Bielawski_constant}, {\sf Prop. 1.2} Bielawski shows that if $Y$ is a hyperk\"ahler manifold with a HKLR-compatible $U(1)$-action then $c = \dim_{\mathbb H} Y = n$. As the Ricci form $\rho(t)$ represents the class of $c_1(Z) = -c_1(K_Z)$ one obtains that
\begin{equation}
    n\frac{d}{dt}[\omega(t)] = c_1(Z)
\end{equation}

By the previous step $c_1(Z) = nc_1(\mathcal O(1))$. Hence $[\omega(t)] = [\omega(t_0)] + (t-t_0)c_1(\mathcal O(1))$. Therefore
$$
c_1(\mathcal O(1)) = \lim\limits_{t\to \infty}\frac{t-t_0}{t} c_1(\mathcal O(1)) = \lim\limits_{t\to\infty}\frac{[\omega(t)] - [\omega(t_0)]}{t} = \lim\limits_{t\to\infty} \frac{1}{t}[\omega(t)]
$$
We have just realized $c_1(\mathcal O(1))$ as a limit of K\"ahler clases. That means that the bundle $TN$ is nef by the very definition of a nef vector bundle (\ref{nef_def}).
\end{proof}

\hfill

We shall now apply \ref{nef_Bielawski} to K\"ahler quotients.

\hfill

\begin{corollary}
\label{nef_mine}
Let $N = V\twice C$ be a K\"ahler quotient of a vector space $V$. Suppose that $N$ is smooth and compact. If the Feix--Kaledin metric on $Y=T^*N$ is complete, then the tangent bundle to $N$ is big and nef.
\end{corollary}

\begin{proof} By \ref{very_long_proof} the Feix--Kaledin metric is defined globally on $Y = T^*N$. If this metric is complete, that is to say $Y = \tilde{Y}$, then by \ref{stein} the moment map $\psi\colon Y\to \R_{\le 0}$ is proper. \ref{nef_Bielawski} completes the proof of the nefness of $TN$.

By \ref{affine} the variety $Y_J = \tilde{Y}_J$ is an affine algebraic variety. The result in \cite{Greb}, {\sf Cor. 4.4} states that this property is sufficient for $TN$ to be big (see also the remark below).
\end{proof}

\hfill

\remark The results in \cite{Greb} are formulated for varieties called {\em canonical extensions}. These varieties are discussed in Appendix to this paper under the name {\em twisted cotangent bundles}. We believe that this name is more preferable in the context of hyperk\"ahler geometry than the name used in \cite{Greb}. In particular, \ref{twistedtwist} states that canonical extensions/twisted cotangent bundles are isomorphic to $Y_J$ (the variety $Y$ with "twisted" complex structure). That is why we can apply \cite{Greb} in the proof of \ref{nef_mine}.

\hfill

The property of a variety to have the tangent bundle which is big and nef is quite restrictive. By \cite{Hsiao} the varieties with this property are necessarily Fano. The only known examples of such varieties are rational homogeneous. A famous conjecture by Campana--Peternell (\cite{CP}) predicts that rational homogeneous varieties exhaust the class of Fano varieties with nef tangent bundle.

\hfill

%%%%%%%%%%%%%%%%%%%
%%%%%%%%%%%%%%%%%%
%%%%%%%%%%%%

\section{Appendix: Feix Construction and Twisted Cotangent Bundles}

Let $X$ be a compact K\"ahler manifold with a fixed K\"ahler form $\omega\in\Lambda^{1,1}(X)$. Consider the set of locally free sheaves $\mathcal E$ over $X$ which are extensions of the form 
\begin{equation}
\label{extension}
    0 \to \Omega X \to \mathcal E \to \mathcal{O}_X \to 0
\end{equation}
Their isomorphism classes are classified by $\Ext^1(\mathcal O_X,\Omega X) \cong H^1(X,\Omega X) \cong H^{1,1}(X)$. The first isomorphism holds by general homological algebra results, the second one follows from the classical Hodge theory.

\hfill

\begin{definition}
\label{twist_def}
\begin{itemize}
    \item[(i)] Consider the extension $\mathcal E$ corresponding to $[\omega]\in H^{1,1}(X) = \Ext^1(\mathcal O_X,\Omega X)$. We shall call $\mathcal E$ the {\em twistor sheaf} of $X$.
    \item[(ii)] Let $E$ be the vector bundle whose sheaf of sections is $\mathcal E$. Consider the map of vector bundles $\rho\colon E\to \underline{\C}$ induced by the surjection in the short exact sequence~(\ref{extension}). Then the preimage in $E$ of the unit section of the trivial bundle $\underline{\C}$ is an affine bundle over $X$ called the {\em twisted cotangent bundle} (\cite{Beilinson_Kazhdan}). We shall denote it by $\Omega X^{tw}$.
\end{itemize}
\end{definition}

\hfill

The manifold $\Tot(E)$ admits a smooth projection $T'\colon \Tot(E)\to\mathbb A^1$ which we define as a composition of the map $\rho\colon \Tot(E)\to X\times \mathbb A^1$ with the projection to the second factor. Its fiber over $0\in\mathbb A^1$ is $T^*X:=\Tot(\Omega X)$. Hence we constructed a deformation of $T^*X$ over $\mathbb A^1$.

We shall now rewrite the Feix construction (\cite{_Feix1_}) of the Feix--Kaledin hyperk\"ahler structure on a neighbourhood $U$ of the zero section of $T^*X$ in more canonical terms. That will allow us to see that her construction is deeply related with twistor sheaves defined above. This relation is not transparent from her papers.

Let $X$ be a complex manifold. Then the manifold $X^\C := X\times\overline{X}$ is the canonical complexification of the real manifold $X_\R$. Indeed, the manifold $X_\R$ admits a totally real embedding
\begin{equation}
    \Delta\colon X \to X\times \overline{X}\:\:\:\:\: x\mapsto (x,x)
\end{equation}
Denote by $\pi\colon X\times\overline{X}\to X$ the projection to the first factor.

Now let $\omega\in\Lambda^{1,1}(X)$ be a real-analytic K\"ahler form on $X$. The form $\omega$ can be extended to the holomorphic $2$-form $\omega^\C$ on a neighbourhood $U$ of $\Delta(X)$ in $X^\C$. Let $(z_1,z_2,...z_n)$ be holomorphic coordinates on $X$ such that
$$
\omega = \sum\limits_{i,j} h_{ij}dz_i\wedge d\overline{z_j}
$$
where $h_{ij}$ are local real analytic function on $X$. Then $(z_1,... z_n, w_1,.. w_n)$ where $w_i := \overline{z_i}$ is a holomorphic coordinate system on $X\times \overline{X}$. The complexified form $\omega^\C$ is locally given as
$$
\omega^\C = \sum\limits_{i,j} \tilde{h_{ij}} dz_i\wedge dw_j
$$
where $\tilde{h_{ij}}$ is the holomorphic extension of $h_{ij}$. We shall abuse the notation and denote the neighbourhood $U$ on which $\omega^\C$ is defined again by $X^\C$ though $U$ may be smaller.

The coordinate description allows one to see clearly that the projection $\pi\colon X^\C\to X$ is a holomorphic Lagrangian fibration with smooth fibers. Let $\mathcal T({X^\C/X})$ and $\Omega({X^\C/X})$ be the relative tangent and cotangent bundle respectively of $X^\C$ with respect to $X$. We define a holomorphic fiberwise connection $\nabla\colon \mathcal T({X^\C/X}) \to \mathcal T({X^\C/X})\otimes\Omega({X^\C/X})$ on the fibers of the fibration $\pi$ by the following formula
\begin{equation}
\label{affine_connection_def}
(\nabla_V U)\cntrct\omega^\C = V\cntrct(d (U\cntrct \omega^\C))
\end{equation}
for every holomorphic vector fields $V,U$ tangent to the fibers. We shall call $\nabla$ the {\em Liouville--Arnold} connection. The following fact is well-known (\cite{_Feix1_}).

\hfill

\begin{proposition} Liouville--Arnold connection $\nabla\colon \mathcal T({X^\C/X}) \to \mathcal T({X^\C/X})\otimes\Omega({X^\C/X})$ is well-defined by formula~(\ref{affine_connection_def}). Furthermore, it is flat and torsion-free along the fibers.
\end{proposition}\endproof{}

\hfill

Let $\mathcal F_0$ be a subsheaf of $\pi_* \Omega({X^\C/X})$ consisting of $1$-forms parallel with respect to $\nabla$. The flatness of $\nabla$ implies that $\mathcal F_0$ is locally free. 

\hfill

\begin{proposition}
\label{tangent_and_parallel}
The locally free sheaf $\mathcal F_0$ on $X$ is naturally isomorphic to $\mathcal TX$.
\end{proposition}

\begin{proof} A $1$-form $\alpha\in\Omega({X^\C/X})$ is parallel if and only if for every vector fields $V, U$ tangent to the fibers
\begin{equation}
    (\nabla_V U)\cntrct \alpha = V(U\cntrct\alpha)
\end{equation}

Let $\chi $ be a holomorphic vector field on $X$. It defines a vector filed $\chi_0 := (\chi,0)$ on $\pi^{-1}(U)\subset X^\C\subset X\times \overline{X}$. Then the form $\alpha:=\chi_0\cntrct\omega^\C$ is $\nabla$-parallel. Indeed,

\begin{gather*}
    (\nabla_V U)\cntrct\alpha = -\omega^\C(\nabla_V U,\chi_0) = d(U\cntrct\omega^\C)(\chi_0, V) = \\
    = \chi_0( \omega^\C(U,V)) - V(\omega^\C(U,\chi_0) - \omega^\C (U,[\chi_0,V])
\end{gather*}
The second equality follows from the definition of $\nabla$ and the third one from the Cartan formula for the de Rham differential. The term $\chi_0(\omega^\C(U,V))$ vanishes as the fibers are Lagrangian. The term $\omega^\C(U,[\chi_0,V])$ vanishes for the same reason because $[\chi_0, V]$ is tangent to the fibers. Hence
$$
(\nabla_V U)\cntrct\alpha = V(\omega^\C(\chi_0,U)) = V(U\cntrct\alpha)
$$
and $\alpha$ is indeed parallel.

The map of locally free sheaves $\mathcal TX\to \mathcal F_0$ sending $\chi$ to $\alpha$ is injective as the $2$-form $\omega^\C$ is non-degenerate. The equality of the ranks of $\mathcal TX$ and $\mathcal F_0$ implies that this map is an isomorphism.
\end{proof}

\hfill

The vanishing of the torsion of $\nabla$ implies that every parallel form is closed. Shrinking $X^\C$ if necessary we may assume that the fibers of $\pi\colon X^\C\to X$ are simply connected. Consider the map of sheaves of $\mathcal O_X$-modules
$$
d\colon \pi_*\mathcal O_{X^\C/X} \to \Omega({X^\C/X})
$$
where $d$ is the de Rham differential. Notice that it is indeed a map of sheaves of $\mathcal O_X$-modules. We define $\mathcal F\subset \pi_*\mathcal O_{X^\C/X}$ to be the preimage of $\mathcal F_0\subset \Omega({X^\C/X}) $ by the map $d$.

By the assumption of simply-connectedness of the fibers the map $d\colon \mathcal F\to \mathcal F_0\cong\mathcal TX$ is surjective with the kernel $\mathcal O_X$. We obtain an extension
\begin{equation}
\label{extension_Feix}
0 \to \mathcal O_X \to \mathcal F \to \mathcal TX\to 0
\end{equation}

\hfill

\remark This extension is not in general non-trivial as an extension of locally free sheaves of $\mathcal O_X$-modules. Nevertheless, there exists a {\em real} splitting of the short exact sequence~(\ref{extension_Feix}). We have a $C^\infty(X)$-linear map
\begin{equation}
    ev_\Delta\colon C^\infty(\mathcal F) \to C^\infty(X)\:\:\:\:\:\:\:\:\:\: ev_\Delta (f):= f\restrict{\Delta(X)}
\end{equation}
which is the desired splitting.

\hfill

It will be more convenient for us to work with the dualiziation of the short exact sequence~(\ref{extension_Feix})
\begin{equation}
\label{extension_Feix_dual}
    0 \to \Omega X \to \mathcal F^\vee \to \mathcal O_X\to 0
\end{equation}

\hfill

\begin{theorem}\label{very_nice_theorem} The sheaf $\mathcal F^\vee$ is the twistor sheaf of $X$ (\ref{twist_def}).
\end{theorem}

\pstep The cohomology of the sheaf $\Omega X$ can be computed by two different complexes. The first one is the $\overline{\partial}$-complex
\begin{equation}
    \Lambda^{1,\bullet}(X) = \{\:0\to\Lambda^{1,0}(X) \to \Lambda^{1,1}(X) \to \Lambda^{1,2}(X)\to...\:\}
\end{equation}
Now let us choose a locally finite covering of $X$ by affine charts so that $X = \bigcup\limits_i U_i$. Then the \v{C}ech complex $\check{C}^\bullet(\Omega X)$ assosiated with this covering also computes the cohomology of $\Omega X$. This complex is given explicitly as
\begin{equation}
    \check{C}^\bullet(\Omega X) = \{\:0\to \bigoplus\limits_{i} \Omega X(U_i) \to \bigoplus\limits_{i<j}\Omega X(U_{ij})\to\bigoplus\limits_{i<j<k}\Omega X(U_{ijk})\to...\:\}
\end{equation}
where for any subset $I$ of indices we denote $\bigcap\limits_{i\in I} U_i$ by $U_I$. 

Let $[\omega]\in H^1(\Lambda^{1,\bullet}(X))$ be a cohomology class represented by the closed K\"ahler form $\omega\in\Lambda^{1,1}(X)$. Let $\omega_i$ be the restriction of $\omega$ to $U_i$. On every open subset $U_i$ choose a K\"ahler potential $f_i$ i.e. the function $f_i\colon U_i\to\R$ such that
\begin{equation}
    \sqrt{-1}\partial\overline{\partial}f_i = \omega_i
\end{equation}
For every pair of indices $i<j$ consider the holomorphic form $\eta_{ij}:=\sqrt{-1}\partial(f_i\restrict{U_{ij}} - f_j\restrict{U_{ij}})\in\Omega X(U_{ij})$. We claim that the element $(\eta_{ij})\in \bigoplus\limits_{i<j}\Omega X(U_{ij})$ represents the same cohomology class in $H^1(X,\Omega X)$ as $[\omega]$. That follows from standard homological algebra arguments. More precisely, one first identifies the cohomology groups of $\Lambda^{1,\bullet}(X)$ and $\check{C}^\bullet(\Omega X)$ with the cohomology groups of the total complex of the following bicomplex
\begin{equation}
    A^{p,q} = \bigoplus\limits_{|I|=q}\Omega^{1,p}(U_I)
\end{equation}
and then use a simple diagram chasing. Details can be found for example in Stacks project (\cite{Stacks_project} Part 1, Ch. 12, Section 12.25).

{\bf Step 2:} As the K\"ahler form $\omega$ is real analytic, a local K\"ahler potential $f_i\colon U_i\to\R$ may also be chosen to be real analytic. Let $\tilde{f_i}$ be its holomorphic extension to $U_i^\C\subset U_i\times \overline{U_i}$. Then
\begin{equation}
    \omega^\C\restrict{U_i^\C} = \sqrt{-1}d_0 d_1 \tilde{f_i}
\end{equation}
where $d_0$ is the component of $d$ in the direction of $U_i$ and $d_1$ is the component of $d$ in the direction of $\overline{U_i}$. By the proof of \ref{tangent_and_parallel} the map sending a holomorphic vector field $\chi$ on $X$ to $\chi_0\cntrct\omega^\C$ induces an isomorphism between the tangent bundle $\mathcal TX$ and the bundle $\mathcal F_0$ of parallel $1$-forms on the fibers of $\pi\colon X^\C\to X$. Locally
\begin{gather*}
    \chi_0\cntrct \omega^\C = \sqrt{-1} \chi_0\cntrct d_0d_1\tilde{f_i} = \sqrt{-1} \chi_0\cntrct d d_1\tilde{f_i} = \\
    = \sqrt{-1} \mathsf L_{\chi_0}d_1 \tilde{f_i} = \sqrt{-1} d_1(\chi_0(\tilde{f_i})) = d_1 (\sqrt{-1}\tilde{\chi(f_i)})
\end{gather*}
We thus obtain a local holomorphic splitting $\iota_i\colon \mathcal TX\to \mathcal F$ of the short exact sequence~(\ref{extension_Feix})
\begin{equation}
\label{splitting}
    \iota_i(\chi) := \sqrt{-1}\tilde{\chi(f_i)}
\end{equation}

{\bf Step 3:} The short exact sequence~(\ref{extension_Feix_dual}) induces the long exact sequence of cohomology groups
\begin{equation}
    0 \to H^0(X,\Omega X) \to H^0(X,\mathcal F^\vee) \to H^0(X,\mathcal O_X) \to H^1(X,\Omega X)\to...
\end{equation}
By a standard result from homological algebra the class of the extension $\mathcal F^\vee$ in $H^1(X,\Omega X)$ is the image of the unit section $\mathbb 1_X\in H^0(X,\mathcal O_X)$ in $H^1(X,\Omega X)$. We compute it in terms of the \v{C}ech cohomology. One has a complex of short exact sequences

$$
\begin{CD}
0 @>>> \bigoplus\limits_{i}\Omega X(U_{i}) @>>> \bigoplus\limits_{i}\mathcal F^\vee(U_{i}) @>>> \bigoplus\limits_{i}\mathcal O_X(U_{i}) @>>> & 0\\
@. @VVV @VVV @VVV @.\\
0 @>>> \bigoplus\limits_{i<j}\Omega X(U_{ij}) @>>> \bigoplus\limits_{i<j}\mathcal F^\vee(U_{ij}) @>>> \bigoplus\limits_{i<j}\mathcal O_X(U_{ij}) @>>> 0
\end{CD}
$$
The morphism $H^0(X,\mathcal O_X)\to H^1(X,\Omega X)$ is given explicitly as follows. In Step 2 we constructed a local splitting of the short exact sequence~(\ref{extension_Feix_dual}). Hence we have a lift of the element $(\mathbb 1\restrict{U_i})\in\bigoplus\limits_i\mathcal O(U_i)$ to an element of $\bigoplus\limits_{i}\mathcal F^\vee(U_{i})$. We shall denote its component in $\mathcal F^\vee(U_i)$ by $\mathbb 1_i$. It is uniquely defined by the properties
\begin{gather*}
    \mathbb 1_i(h) = h
\end{gather*}
for every holomorphic function $h$ on $X$, and
\begin{gather*}
    \mathbb 1_i(\tilde{\chi(f_i)}) = 0
\end{gather*}
for any holomorphic every vector field $\chi$ on $X$.

The downward pointing differential sends $(\mathbb 1_i)\in \bigoplus\limits_{i}\mathcal F^\vee(U_{i})$ to $(\mathbb 1_j\restrict{U_{ij}} - \mathbb 1_i\restrict{U_{ij}}) \in \bigoplus\limits_{i<j}\mathcal F^\vee(U_{ij})$. This last element in fact lies in the image of $\bigoplus\limits_{i}\Omega X(U_{i})$. Let $(\eta_{ij}')\in \bigoplus\limits_{i}\Omega X(U_{i})$ be its preimage. Then for every local vector field $\chi$ on $X$
\begin{gather*}
    \eta_{ij}'(\chi) = (\mathbb 1_j\restrict{U_{ij}} - \mathbb 1_i\restrict{U_{ij}}) (\iota_i(\chi)) = \sqrt{-1}(\mathbb 1_j\restrict{U_{ij}} - \mathbb 1_i\restrict{U_{ij}}) (\tilde{\chi(f_i)}) = \\
    = \sqrt{-1} \mathbb 1_j\restrict{U_{ij}}(\tilde{\chi(f_i)} - \tilde{\chi(f_j)}) = \sqrt{-1}\chi (f_i - f_j)
\end{gather*}
where the last identity holds as the function $\chi(f_i) - \chi(f_j)$ is a holomorphic function on $X$. We now readily see that $\eta_{ij}'= \sqrt{-1}\partial(f_i - f_j)$. By the results of Step 1 of the proof $\eta_{ij}'=\eta_{ij}$. Consequently, the element $(\eta_{ij}')\in \bigoplus\limits_{i}\Omega X(U_{i})$ represents the same cohomology class in $H^1(X,\Omega X)$ as the K\"ahler form $\omega$. By \ref{twist_def} the locally free sheaf $\mathcal F^\vee$ is the twistor sheaf.
\endproof{}

\hfill

\begin{corollary} 
\label{twistedtwist}
For a small enough neighbourhood $U$ of the zero section of $T^*M$ there exists a hyperk\"ahler metric on $U$  with the following property. Let $T\colon \Tw(U)\to\mathbb P^1$ be the twistor deformation. Let $\mathcal E$ be the twistor sheaf of $X$ and $E$ be the corresponding vector bundle. Then their exist a holomorphic open embedding $u\colon T^{-1}(\mathbb A^1)\to\Tot(E)$ commuting with the projection to $\mathbb A^1$. In other words, the following diagram commutes.

$$
\xymatrix{
\Tot(E) \ar[d]^{T'} & \:T^{-1}(\mathbb A^1)\: \ar@{_{(}->}[l]_u \ar[d] \ar@{^{(}->}[r] & \Tw(U) \ar[d]^T\\
\mathbb A^1 \ar@{=}[r] & \mathbb A^1\: \ar@{^{(}->}[r] & \mathbb P^1
}
$$
\end{corollary}
\begin{proof} Feix in \cite{_Feix1_} constructs a complex manifold $Z$ with a projection $T\colon Z\to \mathbb P^1$ which turns out to be the twistor projection for some hyperk\"ahler metric on $U$. She obtains $Z$ by gluing the complex manifolds $\Tot(F^*)$ and $\Tot(\bar{F^*})$ over $\P^1$ by means of a certain holomorphic map. In particular, by the very construction of her the manifold $\Tot(F^*)$ is isomorphic to $T^{-1}(\mathbb A^1)$. The statement of the corollary now follows from \ref{very_nice_theorem}.
\end{proof}

\hfill

%%%%%%%%%%%%%%%%%%%%%%%%%%%%%%%%%%%

{\small 
\noindent {\sc Anna Abasheva\\
{\sc Laboratory of Algebraic Geometry, HSE University\\
Department of Mathematics, Usacheva ul., 6, Moscow, Russia}\\
\\
and:\\
{\sc Columbia University, Department of Mathematics,\\
2990 Broadway, New York, NY 10027}\\
\\
also:\\
{\sc Independent University of Moscow\\
Bolshoy Vlasievskiy per., 11, Moscow, Russia}\\
\tt  anabasheva(at)yandex.ru}.
 }

\end{document}